\documentclass{amsart}

\usepackage[T1]{fontenc}
\usepackage[english]{babel}
\usepackage{amsfonts,amsmath,amssymb}
\usepackage{amsthm}
\usepackage[mathscr]{eucal}
\usepackage{tikz}
\usetikzlibrary{calc}
\usepackage[hang]{subfigure}
\usepackage{xspace}
\usepackage{booktabs,tabularx}
\definecolor{links}{rgb}{.2,.1,.5}
\definecolor{cites}{rgb}{.5,.1,.2}
\usepackage[colorlinks,linkcolor=links,citecolor=cites]{hyperref}
\usepackage{aliascnt}
\usepackage{varioref}
\usepackage[colorinlistoftodos, bordercolor=orange, backgroundcolor=orange!20, linecolor=orange, textsize=scriptsize]{todonotes}

\usetikzlibrary{3d}

\newcommand\CC{{\mathbb C}}

\newcommand\QQ{{\mathbb Q}}
\newcommand\RR{{\mathbb R}}
\newcommand\ZZ{{\mathbb Z}}
\newcommand\PP{{\mathbb P}}
\newcommand{\cH}{{\mathcal H}}

\newcommand{\str}{simplicial, terminal, and reflexive\xspace}

\newcommand{\1}{\mathbf 1}
\newcommand{\0}{\mathbf 0}

\newcommand\neigh{\operatorname{N}} 

\newcommand{\SetOf}[2]{\left\{#1\vphantom{#2}\,\right.\left|\,\vphantom{#1}#2\right\}}
\newcommand{\smallSetOf}[2]{\{#1\,|\,#2\}}

\newcommand\HSBC[1]{\operatorname{DP}(#1)} 
\renewcommand\Vert[1]{\operatorname{Vert}(#1)} 
\DeclareMathOperator\ecc{ecc}
\DeclareMathOperator\conv{conv}
\DeclareMathOperator\pos{pos}

\DeclareMathOperator\lin{lin}
\DeclareMathOperator\opp{opp}
\DeclareMathOperator\Hom{Hom}
\newcommand\GL[2]{\textrm{GL}_{#1}{#2}}

\newcommand{\theoremname}{Theorem}
\newcommand{\corollaryname}{Corollary}
\newcommand{\lemmaname}{Lemma}
\newcommand{\propositionname}{Proposition}
\newcommand{\conjecturename}{Conjecture}
\newcommand{\remarkname}{Remark}
\newcommand{\examplename}{Example}
\newcommand{\definitionname}{Definition}
\newcommand{\questionname}{Question}
\newcommand{\answername}{Answer}

\theoremstyle{plain}
    \newtheorem{theorem}{\theoremname}
    \newtheorem*{theorem*}{\theoremname}
  \newaliascnt{corollary}{theorem}
    \newtheorem{corollary}[corollary]{\corollaryname}
  \newaliascnt{lemma}{theorem}
    \newtheorem{lemma}[lemma]{\lemmaname}
  \newaliascnt{proposition}{theorem}
    \newtheorem{proposition}[proposition]{\propositionname}
\theoremstyle{definition}
  \newaliascnt{conjecture}{theorem}
    \newtheorem{conjecture}[conjecture]{\conjecturename}
  \newaliascnt{remark}{theorem}
    \newtheorem{remark}[remark]{\remarkname}
  \newaliascnt{example}{theorem}
    \newtheorem{example}[example]{\examplename}
  \newaliascnt{definition}{theorem}
    
  \newaliascnt{question}{theorem}
    
  \newaliascnt{answer}{theorem}

\aliascntresetthe{corollary}
\aliascntresetthe{lemma}
\aliascntresetthe{proposition}
\aliascntresetthe{conjecture}
\aliascntresetthe{remark}
\aliascntresetthe{example}
\aliascntresetthe{definition}
\aliascntresetthe{question}
\aliascntresetthe{answer}

\newcommand\Obro{{\O}bro\xspace}

\DeclareMathOperator{\image}{im}

\AtBeginDocument{
  \setlength{\leftmargini}{2em}
  \setlength{\leftmarginii}{2em}
  \setlength{\leftmarginiii}{2em}
  \setlength{\leftmarginiv}{2em}
}

\title{Smooth Fano Polytopes With Many Vertices}

\author{Benjamin Assarf}
\author{Michael Joswig}
\address{Benjamin Assarf, Michael Joswig,\newline \mbox{ }\quad
TU Berlin, Inst.\ Mathematik, MA 6-2, Str. des 17. Juni 136, 10623 Berlin, Germany}
\email{\{assarf,joswig\}@math.tu-berlin.de}

\author{Andreas Paffenholz}
\address{Andreas Paffenholz,\newline \mbox{ }\quad TU Darmstadt, FB Mathematik, Dolivostr.\ 15, 64293 Darmstadt, Germany}
\email{paffenholz@mathematik.tu-darmstadt.de}

\thanks{The second and third authors are supported by the Priority Program 1489 ``Algorithmic and
  Experimental Methods in Algebra, Geometry and Number Theory'' of the German Research Council
  (DFG)}



\begin{document}

\begin{abstract}
  The $d$-dimensional simplicial, terminal, and reflexive polytopes with at least $3d-2$ vertices are classified.  In particular,
  it turns out that all of them are smooth Fano polytopes.  This improves on previous results of
  (Casagrande, 2006) and (\O{}bro, 2008).  Smooth Fano polytopes play a role in algebraic geometry and
  mathematical physics.
\end{abstract}

\maketitle

\section{Introduction}

A \emph{lattice polytope} $P$ is a convex polytope whose vertices lie in a lattice $N$ contained in
the vector space $\RR^d$.  Fixing a basis of $N$ describes an isomorphism to $\ZZ^d$.  Throughout
this paper, we restrict our attention to the standard lattice $N=\ZZ^d$. A $d$-dimensional lattice
polytope $P\subset\RR^d$ is called \emph{reflexive} if it contains the origin $\0$ as an interior
point and its polar polytope is a lattice polytope in the dual lattice $M:=\Hom(N,\ZZ)\cong\ZZ^d$.
A lattice polytope $P$ is \emph{terminal} if $\0$ and the vertices are the only lattice points in
$P\cap \ZZ^d$.  It is \emph{simplicial} if each face is a simplex.  We say that $P$ is a
\emph{smooth Fano polytope} if $P\subseteq \RR^d$ is simplicial with $\0$ in the interior and the
vertices of each facet form a lattice basis of $\ZZ^d$. The fan where every cone is the non negative
linear span over a face is called the \emph{face fan}. It is dual to the \emph{normal fan}, which is
the collection of all normal cones.

In algebraic geometry, reflexive polytopes correspond to \emph{Gorenstein toric Fano varieties}. The
toric variety $X_P$ of a polytope $P$ is determined by the face fan of $P$, that is, the fan spanned
by all faces of $P$; see Ewald \cite{Ewald96} or Cox, Little, and Schenck
\cite{Book_CoxLittleSchenck} for more details.  The toric variety $X_P$ is \emph{$\QQ$-factorial}
(some multiple of a Weil divisor is Cartier) if and only if the polytope $P$ is simplicial.  In this
case the \emph{Picard number} of $X$ equals $n-d$, where $n$ is the number of vertices of $P$.  The
polytope $P$ is smooth if and only if the variety $X_P$ is a manifold (that is, it has no
singularities).  Note that the notions detailed above are not entirely standardized in the
literature.  For example, our definitions agree with~\cite{1067.14052}, but disagree
with~\cite{KN5}.  Our main result, \autoref{thm:main}, is the following.

\begin{theorem*}
  Any $d$-dimensional \str lattice polytope with at least $3d-2$ vertices is lattice equivalent to a
  direct sum of del Pezzo polytopes, pseudo del Pezzo polytopes, or a (possibly skew and multiple)
  bipyramid over (pseudo) del Pezzo polytopes.  In particular, such a polytope is necessarily smooth
  Fano.
\end{theorem*}

This extends results of Casagrande who proved that the number of vertices of $d$-di\-men\-sional \str lattice
polytopes does not exceed $3d$ and who showed that, up to lattice equivalence, only one type exists
which attains this bound (and the dimension $d$ is even)~\cite{Casagrande06}.  Moreover, our result
extends results of \Obro who classified the polytopes of the named kind with $3d-1$ vertices
\cite{Obro08}.  Our proof employs techniques similar to those used by \Obro~\cite{Obro08} and Nill
and \Obro~\cite{NillObro10}, but requires more organization since a greater variety of possibilities
occurs.  One benefit of our approach is that it suggests a certain general pattern emerging, and we
state this as \autoref{conjecture} below.  Translated into the language of toric varieties our main
result establishes that any $d$-dimensional terminal $\QQ$-factorial Gorenstein toric Fano variety
with Picard number at least $2d-2$ decomposes as a (possibly trivial) toric fiber bundle with
known fiber and base space; the precise statement is \autoref{cor:toric}.

The interest in such classifications has its origins in applications of algebraic geometry to mathematical
physics. For instance, Batyrev~\cite{Batyrev94calabi} uses reflexive polytopes to construct pairs of mirror
symmetric Calabi-Yau manifolds; see also Batyrev and Borisov~\cite{BB}. Up to unimodular equivalence, there exists
only a finite number of such polytopes in each dimension, and they have been classified up to dimension $4$, see
Batyrev \cite{Batyrev91}, Kreuzer and Skarke \cite{KreuzerSkarke3d,KreuzerSkarke4d}. Smooth reflexive polytopes
have been classified up to dimension $8$ by \Obro~\cite{OebroPhD}; see \cite{GRDB} for data. By enhancing \Obro's
implementation within the \texttt{polymake} framework \cite{DMV:polymake} this classification was extended to
dimension~$9$ \cite{smoothreflexive}; from that site the data is available in \texttt{polymake} format.

We are indebted to Cinzia Casagrande for helping to improve the description of the toric varieties
associated with the polytopes that we classify.  Finally, we thank an anonymous referee for very
careful reading; her or his comments lead to a number of improvements concerning the exposition.

\section{Lattice Polytopes}

Let $N\cong\ZZ^d$ be a lattice with associated real vector space $N_\RR:=N\otimes_\ZZ\RR$ isomorphic
to $\RR^d$.  The lattice $M:=\Hom_\ZZ(N,\ZZ)\cong \ZZ^d$ is dual to $N$ with dual vector space
$M_\RR=M\otimes_\ZZ\RR\cong \RR^d$.  A polytope $P\subset\RR^d$ is a \emph{lattice polytope} with
respect to $N$ if its\emph{ vertex set} $\Vert{P}$ is contained in $N$.  If the polytope $P$ is
full-dimensional and contains the origin $\0$ as an interior point, then the \emph{polar}
\[
P^* \ = \ \SetOf{w\in\RR^d}{\langle v,w\rangle\le 1 \text{ for all } v\in P}
\]
is a convex $d$-polytope, too, which also contains the origin as an interior point.  We always have
${(P^*)}^*=P$.  However, in general, the polar of a lattice polytope is not a lattice polytope. If $P$
is a lattice polytope in $M$ and the vertices of $P^*$ are contained in $M$, then the polytopes $P$
and $P^*$ are called \emph{reflexive}.  The lattice polytope $P$ is \emph{terminal} if $P\cap
N=\Vert{P}\cup\{\0\}$.  More generally, $P$ is \emph{canonical} if the origin is the only interior
lattice point in $P$.  Two lattice polytopes are \emph{lattice equivalent} if one can be mapped to
the other by a transformation in $\GL{d}{\ZZ}$ followed by a lattice translation.  Throughout the
paper we assume that our polytopes lie in the standard lattice $N=\ZZ^d$.

We start out with listing all possible types of $2$-dimensional terminal and reflexive lattice
polytopes in \autoref{fig:2dmin_Fano}.  Up to lattice equivalence five cases occur which we denote
as $P_6$, $P_5$, $P_{4a}$, $P_{4b}$, and $P_3$, respectively; one hexagon, one pentagon, two
quadrangles, and a triangle; see Ewald \cite[Thm.~8.2]{Ewald96}.  All of them are \emph{smooth Fano}
polytopes, that is, the origin lies in the interior and the vertex set of each facet forms a lattice
basis.  The only $1$-dimensional reflexive polytope is the interval $[-1,1]$.  See Cox et
al.~\cite[p.~382]{Book_CoxLittleSchenck} for the classification of all $2$-dimensional reflexive
polytopes (of which there are 16 types).
\begin{figure}[tb]
  \centering
  \subfigure[{$P_6$}]{
		\begin{tikzpicture}[scale=0.9]
		  \tikzstyle{edge} = [draw,thick,-,black]
		  
		  \foreach \x in {-1,0,1}
		    \foreach \y in {-1,0,1}
		       \fill[gray] (\x,\y) circle (1.5pt); ;
		
		  \coordinate (v0) at (0,0);
		  \coordinate (e1) at (1,0);
		  \coordinate (e2) at (0,1);
		  \coordinate (-e1+e2) at (-1,1);
		  \coordinate (-e2+e1) at (1,-1);
		  \coordinate (-e1) at (-1,0);
		  \coordinate (-e2) at (0,-1);
		  
		  \draw[edge] (e1) -- (e2) -- (-e1+e2) -- (-e1) -- (-e2) -- (-e2+e1) -- (e1);
		  
		  \foreach \point in {e1,e2,-e1+e2,-e2+e1,-e1,-e2}
		    \fill[black] (\point) circle (2pt);
		
		\end{tikzpicture}
  }
  \hfill
  \subfigure[{$P_5$}]{
		\begin{tikzpicture}[scale=0.9]
		  \tikzstyle{edge} = [draw,thick,-,black]
		
		  \foreach \x in {-1,0,1}
		    \foreach \y in {-1,0,1}
		       \fill[gray] (\x,\y) circle (1.5pt); ;
		
		  \coordinate (v0) at (0,0);
		  \coordinate (e1) at (1,0);
		  \coordinate (e2) at (0,1);
		  \coordinate (-e1+e2) at (-1,1);
		  \coordinate (-e2+e1) at (1,-1);
		  \coordinate (-e2) at (0,-1);
		  
		  \draw[edge] (e1) -- (e2) -- (-e1+e2) -- (-e2) -- (-e2+e1) -- (e1);
		  
		  \foreach \point in {e1,e2,-e1+e2,-e2+e1,-e2}
		    \fill[black] (\point) circle (2pt);
		  
		\end{tikzpicture}
  }
  \hfill
  \subfigure[{$P_{4a}$}\label{fig:P4a}]{
		\begin{tikzpicture}[scale=0.9]
		  \tikzstyle{edge} = [draw,thick,-,black]
		
		  \foreach \x in {-1,0,1}
		    \foreach \y in {-1,0,1}
		       \fill[gray] (\x,\y) circle (1.5pt); ;
		
		  \coordinate (v0) at (0,0);
		  \coordinate (e1) at (1,0);
		  \coordinate (e2) at (0,1);
		  \coordinate (-e1) at (-1,0);
		  \coordinate (-e2) at (0,-1);
		
		  \draw[edge] (e1) -- (e2) -- (-e1) -- (-e2) -- (e1);
		  
		  \foreach \point in {e1,e2,-e1,-e2}
		    \fill[black] (\point) circle (2pt);
		\end{tikzpicture}
  }
  \hfill
  \subfigure[{$P_{4b}$}\label{fig:P4b}]{
		\begin{tikzpicture}[scale=0.9]
		  \tikzstyle{edge} = [draw,thick,-,black]
		
		   \foreach \x in {-1,0,1}
		    \foreach \y in {-1,0,1}
		       \fill[gray] (\x,\y) circle (1.5pt); ;
		
		  \coordinate (v0) at (0,0);
		  \coordinate (e1) at (1,0);
		  \coordinate (e2) at (0,1);
		  \coordinate (-e2+e1) at (1,-1);
		  \coordinate (-e1) at (-1,0);
		  
		  \draw[edge] (e1) -- (e2) -- (-e1) -- (-e2+e1) -- (e1);
		  
		  \foreach \point in {e1,e2,-e2+e1,-e1}
		    \fill[black] (\point) circle (2pt);
		\end{tikzpicture}
  }
  \hfill
  \subfigure[{$P_{3}$}]{
		\begin{tikzpicture}[scale=0.9]
		  \tikzstyle{edge} = [draw,thick,-,black]
		  
		  \foreach \x in {-1,0,1}
		    \foreach \y in {-1,0,1}
		       \fill[gray] (\x,\y) circle (1.5pt); ;
		
		  \coordinate (v0) at (0,0);
		  \coordinate (e1) at (1,0);
		  \coordinate (e2) at (0,1);
		  \coordinate (-e1-e2) at (-1,-1);
		
		  \draw[edge] (e1) -- (e2) -- (-e1-e2) -- (e1);
		  
		  \foreach \point in {e1,e2,-e1-e2}
		    \fill[black] (\point) circle (2pt);
		\end{tikzpicture}
  }
  \caption{The $2$-dimensional reflexive and terminal lattice polytopes}
  \label{fig:2dmin_Fano}
\end{figure}
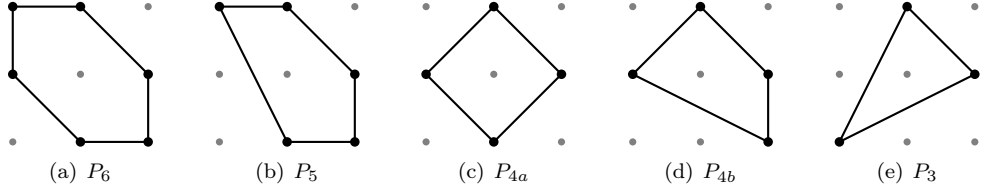

Let $P\subset\RR^d$ and $Q\subset\RR^e$ be polytopes with the origin in their respective relative
interiors. Then the polytope
\[
P\oplus Q \ = \ \conv(P\cup Q) \quad \subset \ \RR^{d+e}
\]
is the \emph{direct sum} of $P$ and $Q$.  This construction also goes by the name ``linear join'' of
$P$ and $Q$.  Clearly, forming direct sums is commutative and associative.  Notice that the polar
polytope $(P\oplus Q)^*=P^*\times Q^*$ is the direct product.  An important special case is the
\emph{proper bipyramid} $[-1,1]\oplus Q$ over $Q$.  More generally, we call a polytope $B$ a
\emph{(skew) bipyramid} over $Q$ if $Q$ is contained in an affine hyperplane $H$ such that there are
two vertices $v$ and $w$ of $B$ which lie on either side of $H$ such that $B=\conv(\{v,w\}\cup Q)$
and the line segment $[v,w]$ meets $Q$ in its (relative) interior.  The relevance of these
constructions for \str polytopes stems from the following three basic facts; see also \cite[\S
V.7.7]{Ewald96} and Figure~\ref{fig:bipyramids} below.  We include a short proof for the sake of
completeness.

\begin{lemma}\label{lemma:Fano+Fano=Fano}
  Let $P\subset\RR^d$ and $Q\subset\RR^e$ both be lattice polytopes. Then the direct sum $P\oplus
  Q\subset\RR^{d+e}$ is simplicial, terminal, or reflexive if and only if $P$ and $Q$ are.
\end{lemma}

\begin{proof}
  We check the three properties separately.  Faces of a direct sum are convex hulls of faces of the
  summands, so $P\oplus Q$ is simplicial if and only if $P$ and $Q$ are.

  The direct sum $P\oplus Q$ contains $\0$ by construction. The direct product $(P\oplus Q)^* =
  P^*\times Q^*$ is integral if and only if both $P^*$ and $Q^*$ are, so if and only if $P$ and $Q$
  are reflexive.

  Finally, if $P$ is not terminal, then there is a non-zero lattice point $x$ in $P$ which is not a
  vertex. By definition of $P\oplus Q$ we have $(x,0)\in P\oplus Q$, and this is not a vertex. So
  the sum is not terminal. Conversely, if $P\oplus Q$ is not terminal, then there is $y\in (P\oplus
  Q)\cap Z^{d+e}\setminus\{\0\}$ which is not a vertex of $P\oplus Q$. We can write $y=\lambda p+\mu
  q$ for points $p\in P$, $q\in Q$ and $\lambda, \mu\ge 0$ with $\lambda+\mu=1$. We can assume that
  $p,q\ne 0$. But $p$ and $q$ are in orthogonal subspaces, so $\lambda p\in P\cap \ZZ^d$ and $\mu
  q\in Q\cap \ZZ^e$. So $\lambda=1$ and $\mu=0$, or vice versa.
\end{proof}

Since the interval $[-1,1]$ is \str we arrive at the following direct consequence.

\begin{lemma}\label{lem:Fano=Bipyramid_over_Fano}
  Let $P=[-1,1]*Q$ be a proper bipyramid over a $(d{-}1)$-dimensional lattice polytope $Q$.  Then
  $P$ is simplicial, terminal or reflexive if and only if $Q$ has the corresponding property.
\end{lemma}

\begin{figure}[t]
  \centering
  \subfigure[{Proper bipyramid over $P_6$}]{
		\begin{tikzpicture}[x  = {(0cm,1.2cm)},
		                    y  = {(2cm,0cm)},
		                    z  = {(1cm,.4cm)},
		                    scale = 1,
		                    color = {lightgray}]
		\tikzset{facestyle/.style={fill=green!20, opacity=.8,draw=black,line join=round}}
		
		\tikzstyle{every label}=[black]
		
		\draw[white] (1.5,0,0) -- (-1.5,0,0);
		
		  \draw[facestyle] (1,0,0) -- (0,1,0) -- (0,0,1) -- (1,0,0) -- cycle ;
		  \draw[facestyle] (1,0,0) -- (0,-1,1) -- (0,0,1) -- (1,0,0) -- cycle ;
		  \draw[facestyle] (1,0,0) -- (0,-1,1) -- (0,-1,0) -- (1,0,0) -- cycle ;
		
		  \draw[facestyle] (-1,0,0) -- (0,1,0) -- (0,0,1) -- (-1,0,0) -- cycle ;
		  \draw[facestyle] (-1,0,0) -- (0,-1,1) -- (0,0,1) -- (-1,0,0) -- cycle ;
		  \draw[facestyle] (-1,0,0) -- (0,-1,1) -- (0,-1,0) -- (-1,0,0) -- cycle ;
		
		    \draw (0,0,1) node [label=right:\raisebox{3ex}{$\scriptstyle e_3$}] {};
		    \fill[black] (0,0,1) circle (2pt);
		    \fill[black] (0,-1,1) circle (2pt);
		
		  \draw[facestyle] (1,0,0) -- (0,-1,0) -- (0,0,-1) -- (1,0,0) -- cycle ;
		  \draw[facestyle] (1,0,0) -- (0,1,-1) -- (0,0,-1) -- (1,0,0) -- cycle ;
		  \draw[facestyle] (1,0,0) -- (0,1,-1) -- (0,1,0) -- (1,0,0) -- cycle ;

		  \draw[facestyle] (-1,0,0) -- (0,-1,0) -- (0,0,-1) -- (-1,0,0) -- cycle ;
		  \draw[facestyle] (-1,0,0) -- (0,1,-1) -- (0,0,-1) -- (-1,0,0) -- cycle ;
		  \draw[facestyle] (-1,0,0) -- (0,1,-1) -- (0,1,0) -- (-1,0,0) -- cycle ;
		
		  \draw (1,0,0) node [label=above:$\scriptstyle e_1$] {};
		  \draw (-1,0,0) node [label=below:$\scriptstyle -e_1$] {};
		  \draw (0,1,0) node [label=right:$\scriptstyle e_2$] {};
		  \fill[black] (1,0,0) circle (2pt);
		  \fill[black] (-1,0,0) circle (2pt);
		  \fill[black] (0,1,0) circle (2pt);
		  \fill[black] (0,-1,0) circle (2pt);
		  \fill[black] (0,1,-1) circle (2pt);
		  \fill[black] (0,0,-1) circle (2pt);
		
		\end{tikzpicture} 
  }
  \subfigure[{Skew bipyramid over $P_6$}]{
		\begin{tikzpicture}[x  = {(0cm,1.2cm)},
		                    y  = {(2cm,0cm)},
		                    z  = {(1cm,.4cm)},
		                    scale = 1,
		                    color = {lightgray}]
		\tikzset{facestyle/.style={fill=green!20, opacity=.8,draw=black,line join=round}}
		\tikzstyle{vertex}=[circle,minimum size=3pt,inner sep=0pt, fill=black]
		\tikzstyle{every label}=[black]
		
		\draw[white] (1.5,0,0) -- (-1.5,0,0);
		
		  \draw[facestyle] (1,0,0) -- (0,1,0) -- (0,0,1) -- (1,0,0) -- cycle ;
		  \draw[facestyle] (1,0,0) -- (0,-1,1) -- (0,0,1) -- (1,0,0) -- cycle ;
		  \draw[facestyle] (1,0,0) -- (0,-1,1) -- (0,-1,0) -- (1,0,0) -- cycle ;
		
		  \draw[facestyle] (-1,0,1) -- (0,1,0) -- (0,0,1) -- (-1,0,1) -- cycle ;
		  \draw[facestyle] (-1,0,1) -- (0,-1,1) -- (0,0,1) -- (-1,0,1) -- cycle ;
		  \draw[facestyle] (-1,0,1) -- (0,-1,1) -- (0,-1,0) -- (-1,0,1) -- cycle ;
		
		  \draw[facestyle] (-1,0,1) -- (0,-1,0) -- (0,0,-1) -- (-1,0,1) -- cycle ;
		
		    \draw (0,0,1) node[label=right:\raisebox{3ex}{$\scriptstyle e_3$}] {};
		    \fill[black] (0,0,1) circle (2pt);
		    \fill[black] (0,-1,1) circle (2pt);
		
		  \draw[facestyle] (1,0,0) -- (0,-1,0) -- (0,0,-1) -- (1,0,0) -- cycle ;
		  \draw[facestyle] (1,0,0) -- (0,1,-1) -- (0,0,-1) -- (1,0,0) -- cycle ;
		  \draw[facestyle] (1,0,0) -- (0,1,-1) -- (0,1,0) -- (1,0,0) -- cycle ;

		  \draw[facestyle] (-1,0,1) -- (0,1,-1) -- (0,0,-1) -- (-1,0,1) -- cycle ;
		  \draw[facestyle] (-1,0,1) -- (0,1,-1) -- (0,1,0) -- (-1,0,1) -- cycle ;
		
		  \draw (1,0,0) node[label=above:$\scriptstyle e_1$] {};
		  \draw (-1,0,1) node[label=below:$\scriptstyle e_3-e_1$] {};
		  \draw (0,1,0) node[label=right:$\scriptstyle e_2$] {};
		  \fill[black] (1,0,0) circle (2pt);
		  \fill[black] (-1,0,1) circle (2pt);
		  \fill[black] (0,1,0) circle (2pt);
		  \fill[black] (0,-1,0) circle (2pt);
		  \fill[black] (0,1,-1) circle (2pt);
		  \fill[black] (0,0,-1) circle (2pt);
		
		\end{tikzpicture} 
  }
  \caption{The $3$-dimensional smooth Fano polytopes with $3d-1=8$ vertices.
    Combinatorially, both are bipyramids over~$P_6$.}
  \label{fig:bipyramids}
\end{figure}
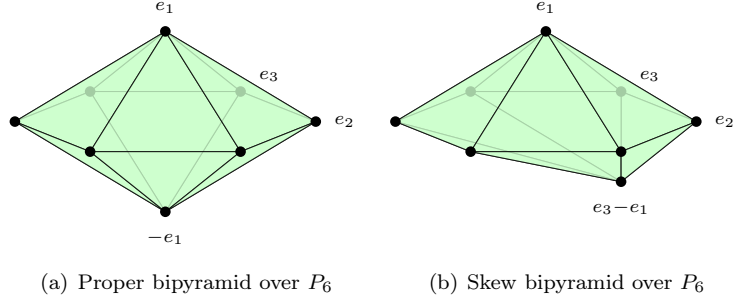

In our proofs below we will frequently encounter the following situation. Let $Q$ be a
$(d{-}1)$-dimensional lattice polytope embedded in $\{\0\}\times\RR^{d-1}\subset \RR^d$, let $e_1$
be the first standard basis vector of $\RR^d$, and let $v$ be a vertex of $Q$ such that the line
segment $\conv\{e_1,v-e_1\}$ intersects $Q$ in the relative interior.  This is to say that
$P:=\conv(\{e_1,v-e_1\}\cup Q)\subset\RR^d$ is a skew bipyramid over $Q$. In this case an argument
similar to the proof of \autoref{lemma:Fano+Fano=Fano} yields a suitable variation of
\autoref{lem:Fano=Bipyramid_over_Fano}.

\begin{lemma}\label{lem:skew-bipyramid}
  The skew bipyramid $P$ is simplicial, terminal or reflexive if and only if $Q$ has the
  corresponding property.
\end{lemma}
\begin{proof}
  As in the proof of Lemma~\ref{lemma:Fano+Fano=Fano} we check the three properties one by one.

  The facets of $P$ are pyramids over facets of $Q$.  A pyramid is a simplex if and only if its base
  is a simplex of one dimension lower.  This means that $P$ is simplicial if and only if $Q$ is.

  A copy of $Q$ arises as the intersection of $P$ with the coordinate subspace, $H$, spanned by
  $e_2,e_3,\dots,e_d$.  Hence the terminality of $P$ implies the terminality of $Q$.  Conversely,
  assume that $Q$ is terminal.  Any non-vertex lattice point of $P$ would have to lie outside the
  hyperplane $H$.  Yet, by construction, the only two vertices outside $H$ are at levels $1$ and
  $-1$; so there is no room for any lattice point in $P$ other than the origin.

  Let us assume that the origin is an interior lattice point of $Q$.  Then each facet $G$ of $Q$ has
  a uniquely determined outer facet normal vector $u_G$ such that an inequality which defines $G$
  reads $\langle u_G,x\rangle\le 1$.  We call this the \emph{standard outer facet normal vector}
  of~$G$. If $F=\conv(G\cup\{v-e_1\})$ is a facet of $P$ which arises as a pyramid over $G$ then the
  standard outer facet normal vector $u_F$ satisfies
  \[
  \langle u_F,e_1\rangle \ = \ \langle u_G,v\rangle - 1 \quad \text{and} \quad
  \langle u_F,e_j\rangle \ = \ \langle u_G,e_j\rangle \text{ for all } j\ge 2 \,.
  \]
  This shows that $u_F$ is integral if and only $u_G$ is.  For the remaining facets of $P$, all of
  which contain $e_1$, the vertex antipodal to $v-e_1$, we have $\langle u_F,e_1\rangle=-1$.
  Therefore, $P$ is reflexive if and only if $Q$ is.
\end{proof}

Let $e_1,e_2, \ldots, e_d$ be the standard basis of $\ZZ^d$ in $\RR^d$. The $d$-polytopes
\[
\HSBC{d} \ = \ \conv\{\pm e_1,\pm e_2, \dots, \pm e_d, \pm \1\} \quad \subset \ \RR^d
\]
for $d$ even with $2d+2$ vertices form a $1$-parameter family of smooth Fano polytopes; see
Ewald~\cite[\S V.8.3]{Ewald96}.  They are usually called \emph{del Pezzo} polytopes.  If $-\1$ is
not a vertex the resulting polytopes are sometimes called \emph{pseudo del Pezzo}.  Here and
throughout we abbreviate $\1=(1,1,\dots,1)$.  Notice that the $2$-dimensional del Pezzo polytope
$\HSBC{2}$ is lattice equivalent to the hexagon $P_6$ shown in \autoref{fig:2dmin_Fano}, and the
$2$-dimensional pseudo del Pezzo polytope is the pentagon $P_5$.  While the definition of $\HSBC{d}$
also makes sense in odd dimensions, these polytopes are not simplicial if $d\ge 3$ is odd.

For \emph{centrally symmetric} smooth Fano polytopes Voskresenski\u\i{} and
Klyachko~\cite{Voskresenskij1985} provide a classification result. They showed that every centrally
symmetric smooth Fano polytope can be written as a sum of line segments and del Pezzo
polytopes. Later Ewald~\cite{Ewald88,Ewald96} generalized this classification for pseudo-symmetric
polytopes. A polytope is \emph{pseudo-symmetric} if there exists a facet $F$, such that $-F =
\SetOf{-v}{v\in F}$ is also a facet.  Nill~\cite{0511294} extended Ewald's result to simplicial
reflexive polytopes.

\begin{theorem}[{Nill~\cite[Thm.~0.1]{0511294}}]\label{thm:pseudo_symmetric}
  Any pseudo-symmetric simplicial and reflexive polytope is lattice equivalent to a direct sum of a
  centrally symmetric reflexive cross polytope, del Pezzo polytopes, and pseudo del Pezzo polytopes.
\end{theorem}

A direct sum of $d$ intervals $[-1,1]\oplus[-1,1]\oplus\dots\oplus[-1,1]$ is the same as the regular
cross polytope $\conv\{\pm e_1,\pm e_2,\dots,\pm e_d\}$, which is centrally symmetric and reflexive.  The direct sum of several intervals with a
polytope $Q$ is the same as an iterated proper bipyramid over $Q$.

\begin{theorem}[{Casagrande~\cite[Thm.~3]{Casagrande06}}]
  A simplicial and reflexive $d$-polytope $P$ has at most $3d$ vertices.  If it does have exactly
  $3d$ vertices then $d$ is even, and $P$ is a centrally symmetric smooth Fano polytope.  Thus in
  this case $P$ is lattice equivalent to a direct sum of $\frac{d}{2}$ copies of $P_6\cong\HSBC{2}$.
\end{theorem}

\Obro classified the \str $d$-polytopes with $3d-1$
vertices~\cite{Obro08}.  We describe the cases which occur, again up to lattice equivalence.  To get
an idea it is instrumental to look at the low-dimensional cases first.  For instance, the interval
$[-1,1]$ is of this kind since it has the right number $3\cdot 1-1=2$ vertices.  In dimension two we
only have the pentagon $P_5$.  In dimension three two cases arise which are combinatorially
isomorphic but inequivalent as lattice polytopes: the proper bipyramid $[-1,1]\oplus P_6$, which has
the vertices
\[
\pm e_1\,,\ \pm e_2\,,\ \pm e_3\,,\ \pm(e_2-e_3)
\]
and the skew bipyramid over $P_6$ with vertices
\[
e_1\,,\ \pm e_2\,,\ \pm e_3\,,\ \pm (e_2-e_3)\,,\ e_3-e_1 \, ;
\]
see \autoref{fig:bipyramids}.  The apices of the proper bipyramid are $\pm e_1$, and in the skew
bipyramid the vertices $e_1$ and $e_3-e_1$ form the apices.

By forming suitable direct sums we can construct more smooth Fano $d$-polytopes with $3d-1$
vertices.  For $d$ even the polytope $P_5 \oplus P_6^{\oplus (\frac{d}{2}-1)}$ is $d$-dimensional,
and it has $3d-1$ vertices.  Up to a lattice isomorphism these can be chosen as follows.
\begin{equation} \label{eq:obro:even}
  \begin{gathered}
    e_1\,,\ \pm e_2\,,\ \ldots\,,\ \pm e_{d}\\
    \pm(e_1-e_2)\,,\ \pm(e_3-e_4)\,,\ \ldots \,,\ \pm (e_{d-1} - e_{d}) \, .
  \end{gathered}
\end{equation}
For an odd number $d$ we can sum up $\frac{d-1}{2}$ copies of $P_6$ and form the bipyramid.  The
resulting $d$-dimensional polytope has the vertices
\begin{equation} \label{eq:obro:odd-not-skew}
  \begin{gathered}
    \pm e_1\,,\ \pm e_2\,,\ \ldots\,,\ \pm e_{d}\\
    \pm(e_2-e_3)\,,\ \pm(e_4-e_5)\,,\ \ldots \,,\ \pm (e_{d-1} - e_{d}) \, .
  \end{gathered}
\end{equation}
Notice that the eight vertices $\pm e_1,\pm e_{d-1},\pm e_d,\pm(e_{d-1}-e_d)$ form a $3$-dimensional
bipyramid $[-1,1]\oplus P_6$.  Replacing this summand by a three dimensional skew bipyramid over
$P_6$ yields a polytope with vertices
\begin{equation} \label{eq:obro:odd-skew}
  \begin{gathered}
    \pm e_1\,,\ \pm e_2\,,\ \ldots\,,\ \pm e_{d-1}\,,\ e_{d}\\
    \pm(e_1-e_2)\,,\ \pm(e_3-e_4)\,,\ \ldots \,,\ \pm (e_{d-2} - e_{d-1})\,,\ (e_1 - e_{d}) \, .
  \end{gathered}
\end{equation}

\begin{theorem}[{\Obro~\cite[Thm.~1]{Obro08}}]\label{thm:3d-1}
  A \str lattice $d$-polytope with exactly $3d-1$ is lattice
  equivalent to the polytope \eqref{eq:obro:even} if $d$ is even, and it is lattice equivalent to
  either \eqref{eq:obro:odd-not-skew} or \eqref{eq:obro:odd-skew} if $d$ is odd.
\end{theorem}

It turns out that all these polytopes are smooth Fano. Note that there even exists a generalization without the
terminality assumption; see Nill and \Obro~\cite{NillObro10}. From this classification it follows that
combinatorially only one case per dimension occurs.  This is a consequence of the fact that a proper bipyramid is
combinatorially equivalent to a skew bipyramid.

Our main result is the following classification.  A \emph{double bipyramid} is a bipyramid over a
bipyramid, and each of these bipyramids can be proper or skew.

\begin{theorem}\label{thm:main}
  Any $d$-dimensional \str polytope $P$ with exactly $3d-2$ vertices
  is lattice equivalent to one of the following.  If $d$ is even then $P$ is lattice equivalent to
  \begin{enumerate}
    \item\label{it:main:even-bipyramid} a double proper or skew bipyramid over $P_6^{\oplus \frac{d-2}{2}}$ or
    \item\label{it:main:even-p5} $P_5^{\oplus 2} \oplus P_6^{\oplus (\frac{d}{2} - 2)}$ or
    \item\label{it:main:even-hsbc4} $\HSBC{4} \oplus P_6^{\oplus (\frac{d}{2} - 2)}$.
  \end{enumerate}
  If $d$ is odd then $P$ is lattice equivalent to
  \begin{enumerate}\setcounter{enumi}{3}
    \item\label{it:main:odd-bipyramid} a proper or skew bipyramid over $P_5 \oplus P_6^{\oplus \frac{d-3}{2}}$.
  \end{enumerate}
  In particular, if $d$ is even there are three combinatorial types, and the combinatorial type is
  unique if $d$ is odd.  Moreover, $P$ is a smooth Fano polytope in all cases.
\end{theorem}

\begin{remark}\label{rem:main}
  The cases above split into several lattice isomorphism classes, which we want to enumerate
  explicitly. For this, observe that up to lattice equivalence any bipyramid over a \str polytope
  $Q\subseteq \RR^{d-1}\times\{0\}$ can be realized as $\conv(Q\cup \{e_d, -e_d+v\})$, where $v$ is
  a vertex of $Q$ or zero. Hence, up to lattice equivalence, each proper or skew bipyramid can be
  described by specifying $v$.  The case $v=0$ is the proper bipyramid. A double proper or skew
  bipyramid is then given by a pair of points.

  Suppose first that $d$ is even. Any even-dimensional \str polytope with $3d-2$ vertices is lattice
  equivalent to $Q\oplus P_6^{\oplus k}$ for a \str polytope $Q$ of (even) dimension at most $6$,
  and $k=\frac{d-\dim Q}{2}$. Hence, it suffices to classify up to dimension $6$. This will be
  explained in the following paragraph.

  The two planar variants are shown in Figures~\ref{fig:P4a} and~\ref{fig:P4b}. For $d=4$ the
  case~\eqref{it:main:even-bipyramid} splits into eight subtypes, depending on the choice of the two
  apices of the bipyramid. Let $v_1, v_2, \ldots, v_6$ be the vertices of $P_6$ in cyclic order, and
  $x,x'$ the apices of the first bipyramid. Then we have the following choices:
  \begin{align*}
    (0,0)\,,\quad (0,x)\,,\quad (0,v_1)\,,\quad (v_1,v_1)\,,\quad (v_1, v_2)\,,\quad
    (v_1,v_3)\,,\quad (v_1,v_4)\;,\quad \text{and}\quad (v_1,x)\,.
  \end{align*}
  The reason is that the group of lattice automorphisms of $P_6$, which is isomorphic to the dihedral
  group of order $12$, acts sharply transitively on adjacent pairs of vertices.  The two
  types~\eqref{it:main:even-p5} and~\eqref{it:main:even-hsbc4} don't split further into lattice
  isomorphism types. In dimension $d=6$ we have one additional choice for the apices of a double
  skew bipyramid: We can choose the vertices in distinct hexagons.  Summarizing, in even dimensions
  there are eleven lattice isomorphism types in dimensions $d\ge 6$, ten for $d=4$ and two for
  $d=2$.

  Now let us look at the odd dimensional cases.  For $d=1$ we have the segment $[-1,1]$, and for
  $d=3$ there are the proper and skew bipyramid over $P_5$. Let $w_1, w_2, \ldots, w_5$ be the
  vertices of $P_5$ in cyclic order, and $w_1$ is the unique vertex such that $-w_1$ is not a vertex
  of $P_5$. For bipyramids in dimensions $d\ge 5$ we can choose $v$ as $0, w_1, w_2, w_3, v_1$,
  which gives us all five possible isomorphism types.  Summarizing, the odd dimensional types split
  into one type for $d=1$, two for $d=3$ and five in all higher dimensions.
\end{remark}

The possible bases of the bipyramid are (up to lattice equivalence) all precisely the
$d{-}1$-dimensional, \str $d$-polytopes with $3d-4=3(d-1)-1$ vertices from \autoref{thm:3d-1}; all
of these are smooth Fano.  We do believe that the list of the classifications obtained so far
follows a pattern.

\begin{conjecture}\label{conjecture}
  Let $P$ be a $d$-dimensional smooth Fano polytope with $n$ vertices such that $n\ge 3d-k$ for
  $k\le\tfrac{d}{3}$.  If $d{+}k$ is even then $P$ is lattice equivalent to $Q\oplus P_6^{\oplus
    (\frac{d-3k}{2})}$ where $Q$ is a $3k$-dimensional smooth Fano polytope with $n-3d+9k\ge
  8k$ vertices.  If $d{+}k$ is odd then $P$ is lattice equivalent to $Q\oplus P_6^{\oplus
    (\frac{d-3k-1}{2})}$ where $Q$ is a $(3k{+}1)$-dimensional smooth Fano polytope with
  $n-3d+9k-3\ge 8k-3$ vertices.
\end{conjecture}
This conjecture is the best possible in the following sense: The $k$-fold direct sum of skew bipyramids
over $P_6$ yields a smooth Fano polytope of dimension $d=3k$ with $8k=3d-k$ vertices which doesn't
admit to split off a single copy of $P_6$ as a direct summand.  However, it does contain
$P_6^{\oplus k}$ as a subpolytope of dimension $2k=\tfrac{2}{3}d$.

If the conjecture above holds the full classification of the smooth Fano polytopes of dimension at
most nine \cite{smoothreflexive} would automatically yield a complete description of all
$d$-dimensional smooth Fano polytopes with at least $3d-3$ vertices.

\section{Toric Varieties}

Regarding a lattice point $a\in\ZZ^d$ as the exponent vector of the monomial
$z^a=z_1^{a_1}z_2^{a_2}\dots z_d^{a_d}$ in the Laurent polynomial ring $\CC[z_1^{\pm 1},z_2^{\pm
  1},\dots,z_d^{\pm 1}]$ provides an isomorphism from the additive group of $\ZZ^d$ to the
multiplicative group of Laurent monomials.  This way the maximal spectrum $X_\sigma$ of a lattice
cone $\sigma$ becomes an \emph{affine toric variety}.  If $\Sigma$ is a fan of lattice cones, gluing
the duals of the cones along common faces yields a \emph{(projective) toric variety} $X_\Sigma$.
This complex algebraic variety admits a natural action of the embedded dense torus corresponding to
(the dual of) the trivial cone $\{\0\}$ which is contained in each cone of $\Sigma$.  If $P\in\RR^d$
is a lattice polytope containing the origin, then the \emph{face fan}
\[
\Sigma(P) \ = \ \SetOf{\pos(F)}{F \text{ face of } P}
\]
is such a fan of lattice cones.  We denote the associated toric variety by $X_P=X_{\Sigma(P)}$.  The
face fan of a polytope is isomorphic to the normal fan of its polar.  Two lattice polytopes $P$ and
$Q$ are lattice equivalent if and only if $X_P$ and $X_Q$ are isomorphic as toric varieties.

Let $P$ be a full-dimensional lattice polytope containing the origin as an interior point.  Then the
toric variety $X_P$ is smooth if and only if $P$ is smooth in the sense of the definition given
above, that is, the vertices of each facet of $P$ are required to form a lattice basis.  A smooth
compact projective toric variety $X_P$ is a \emph{toric Fano variety} if its anticanonical divisor
is very ample.  This holds if and only if $P$ is a smooth Fano polytope; see Ewald
\cite[\S VII.8.5]{Ewald96}.

We now describe the toric varieties arising from the polytopes listed in our \autoref{thm:main}. For the list of
two-dimensional toric Fano varieties we use the same notation as in Figure~\ref{fig:2dmin_Fano}; see Ewald
\cite[\S VII.8.7]{Ewald96}.  The toric variety $X_{P_3}$ is the complex projective plane $\PP_2$. The toric
variety $X_{P_{4a}}$ is isomorphic to a direct product $\PP_1\times\PP_1$ of lines, and $X_{P_{4b}}$ is the smooth
\emph{Hirzebruch surface} $\cH_1$.  The toric variety $X_{P_5}$ is a blow-up of $\PP_2$ at two points or,
equivalently, a blow-up of $\PP_1\times\PP_1$ at one torus invariant point.  The toric varieties associated with
the del Pezzo polytopes $\HSBC{d}$ are called \emph{del Pezzo varieties}; see Casagrande~\cite[\S3]{Casagrande03}
for a detailed description. As a special case the toric variety $X_{P_6}$ is a del Pezzo surface or, equivalently,
a blow-up of $\PP_2$ at three non-collinear torus invariant points. Notice that in algebraic geometry \emph{del
Pezzo varieties} is an ambiguous term.

Two polytope constructions play a role in our classification, direct sums and (skew) bipyramids.  We
want to translate them into the language of toric varieties.  Let $P\subset\RR^d$ and
$Q\subset\RR^e$ both be full-dimensional lattice polytopes containing the origin.  Then the toric
variety $X_{P\oplus Q}$ is isomorphic to the direct product $X_P\times X_Q$.  In particular, for
$P=[-1,1]$ we have that the toric variety
\[
X_{[-1,1]\oplus Q} \ = \ \PP_1 \times X_Q
\]
over the regular bipyramid over $Q$ is a direct product with the projective line $\PP_1\cong
X_{[-1,1]}$.  More generally, the toric variety of a skew bipyramid over $Q$ is a toric fiber bundle
with base space $\PP_1$ and generic fiber $X_Q$; see Ewald \cite[\S VI.6.7]{Ewald96}. An example is the
smooth Hirzebruch surface $\cH_1\cong X_{P_{4b}}$, which is a (projective) line bundle over~$\PP_1$.

In order to translate \autoref{thm:main} to toric varieties we need a few more definitions.  For the
sake of brevity we explain these in polytopal terms and refer to Ewald's monograph \cite{Ewald96}
for the details.  A toric variety $X_P$ associated with a canonical lattice $d$-polytope $P$ is
\emph{$\QQ$-factorial} (or \emph{quasi-smooth}) if $P$ is simplicial; see \cite[\S VI.3.9]{Ewald96}.
In this case the \emph{Picard number} equals $n-d$ where $n$ is the number of vertices of $P$; see
\cite[\S VII.2.17]{Ewald96}.  We call a toric variety, $X$, a \emph{$2$-stage fiber bundle} with a
pair $(Y,Z)$ of base spaces if $X$ is a fiber bundle with base space $Z$ such that the generic fiber
itself is a fiber bundle with base space~$Y$.  We say that $Y$ is the base space of the \emph{first
  stage} while $Z$ is the base space of the \emph{second stage}.  That is, in order to construct $X$
one starts out with the generic fiber of the first stage, then forms a fiber product with the first
base space~$Y$, afterwards takes the resulting space as the new generic fiber to finally form $X$ as
a fiber product with the second base space $Z$.

\begin{corollary}\label{cor:toric}
  Let $X$ be a $d$-dimensional terminal $\QQ$-factorial Gorenstein toric Fano variety with Picard
  number $2d-2$. We assume $d\ge 4$.

  If $d$ is even, then $X$ is isomorphic to
  \begin{enumerate}
    \item a $2$-stage toric fiber bundle such that the base spaces of both stages are projective lines
      and the generic fiber of the first stage is the direct product of $\frac{d-2}{2}$ copies of the
      del Pezzo surface $X_{P_6}$, or
    \item the direct product of two copies of $X_{P_5}$ and $\frac{d}{2}-2$ copies of $X_{P_6}$ or
    \item the direct product of the del Pezzo fourfold $X_{\HSBC{4}}$ and $\frac{d}{2}-2$ copies of
      $X_{P_6}$.
  \end{enumerate}
  If $d$ is odd then $X$ is isomorphic to
  \begin{enumerate}\setcounter{enumi}{3}
    \item a toric fiber bundle over a projective line with generic fiber isomorphic to the direct
      product of $X_{P_5}$ and $\frac{d-3}{2}$ copies of $X_{P_6}$.
  \end{enumerate}
\end{corollary}

All fiber bundles in the preceding result may or may not be trivial.

\begin{remark}
  As for the polytopes in~\autoref{thm:main} we can refine the above cases into equivalence classes
  up to toric isomorphisms. By~\autoref{rem:main} there is one type for $d=1$, two types for $d=2,
  3$, ten for $d=4$, five for any odd dimension $d\ge 5$ and eleven for even dimensions $d\ge 6$. In
  dimensions up to $4$ this has been classified previously~\cite{Batyrev2007,Batyrev91}.
\end{remark}

\section{Terminal, Simplicial, and Reflexive Polytopes}

\subsection{Special Facets and \texorpdfstring{$\eta$}{eta}-Vectors}
\phantomsection\label{subsec:eta}

Let $P\subset\RR^d$ be a reflexive lattice $d$-polytope with vertex set $\Vert{P}$.  In particular,
the origin $\0$ is an interior point.  We let
\[
v_P \ := \ \sum_{v\in \Vert P} v
\]
be the \emph{vertex sum} of $P$.  As $P$ is a lattice polytope, the vertex sum is a lattice point.
Now, a facet $F$ of $P$ is called \emph{special} if the cone $\pos F$ in the fan contains the vertex
sum $v_P$.  Since the fan $\Sigma(P)$ generated by the facet cones is complete, a special facet
always exists.  However, it is not necessarily unique.  For instance, if $P$ is centrally symmetric,
we have $v_P=\0$, and each facet is special.

Since $P$ is reflexive, for each facet $F$ of $P$ there is a unique (outer) facet normal
vector $u_F$ in the dual lattice $M\cong\ZZ^d$ and $\alpha\in \ZZ$ which satisfies the following
\begin{enumerate}
\item $u_F$ is primitive, that is, there is no lattice point strictly between $0$ and $u_F$,
\item the affine hyperplane spanned by $F$ is the set $\smallSetOf{x\in\RR^d}{\langle
    u_F,x\rangle=1}$,
\item and $\langle u_F,x\rangle \le 1$ for all points $x\in P$.
\end{enumerate}
The vector $u_F$ is called the \emph{standard outer normal vector} of $F$.  We define
\begin{align*}
    H(F,k) \ &:= \ \SetOf{x\in\RR^d}{\langle u_F,x\rangle=k}\\
    V(F,k) \ &:= \ H(F,k) \cap \Vert P
\intertext{and}
\eta^F_k \ &:= \  |V(F,k)|
\intertext{for any integer $k\le 1$.  We have}
\Vert{P} \ &\phantom{:}= \bigcup_{k\le 1} V(F,k) \ \subset \ \bigcup_{k\le 1} H(F,k)
\end{align*}
and thus $\eta^F_{1}+\eta^F_{0}+\eta^F_{-1}+\cdots=|\Vert{P}|$ is the number of vertices of $P$.  If
a vertex $v$ is contained in $V(F,k)$ we call the number $k$ the \emph{level} of $v$ with respect to
$F$.  The sequence of numbers $\eta^F=(\eta^F_1,\eta^F_0,\eta^F_{-1},\dots)$ is the
\emph{$\eta$-vector} of $P$ with respect to $F$.  If the choice of the facet is obvious from the
context we omit the upper index $F$ in our notation. Notice that the following lemma does not need
terminality.

\begin{lemma}
  Let $P$ be a simplicial and reflexive polytope. The level of $v_P$ is the same for any special facet of $F$.
\end{lemma}
\begin{proof}
  This is obvious if $v_P=0$.  So suppose otherwise. Let $F_1, F_2, \ldots, F_m$ be the special
  facets of $P$.  By assumption, $v_P$ is
  contained in the cone $C:=\pos(F_1\cap F_2\cap \dots\cap F_m)$.  If $v_1, v_2, \ldots, v_\ell$ are
  the vertices in the common intersection of the special facets, then there are coefficients
  $\lambda_1, \lambda_2, \ldots, \lambda_\ell\ge 0$ such that
  \begin{align*}
    v_P\ =\ \sum_{i=0}^{\ell}\lambda_i v_i\,.
  \end{align*}
  Since $P$ is simplicial the coefficients $\lambda_i$ are unique. We define $k := \sum_{i=1}^{\ell}\lambda_i$.  Note that $\langle u_{F_j},
  v_i\rangle = 1$ for all $1\le j\le m$ and $1\le i\le\ell$, as $P$ is reflexive. Hence,
  \begin{align*}
    \langle u_{F_j}, v_P\rangle\ = \ \langle u_{F_j}, \sum_{i=1}^{\ell}\lambda_i v_i\rangle \ =\
    \sum_{i=1}^{\ell}\lambda_i \langle u_{F_j},  v_i\rangle \ =\ \sum_{i=1}^{\ell}\lambda_i\ =\ k\,.
  \end{align*}
  This means $v_P$ is on level $k$ for all special facets.
\end{proof}
Thus, the common level of $v_P$ for all special facets is an invariant of $P$, which we call the
\emph{eccentricity} $\ecc(P)$.

\begin{example}\label{example:skewbip-P5}
  Consider the $3$-dimensional lattice polytope $A$ (shown in \autoref{fig:skewbip-P5}) with the seven vertices
  \[
  \begin{gathered}
    e_1 \,,\ e_2 \,,\ e_3 \,;\\
    e_1-e_2 \,,\ e_2-e_1 \,,\ e_1-e_3 \,;\\
    -e_1 \, .
  \end{gathered}
  \]
  The vertex sum equals $e_1+e_2$, and the two facets
  \[
  F \ := \ \conv\{e_1,\,e_2\,,e_3\} \quad \text{and} \quad G \ := \ \conv\{e_1,\,e_2,\,e_1-e_3\}
  \]
  are special, the remaining eight facets are not.  We have $u_F=\1$ and $u_G=e_1+e_2=\1-e_3$. The
  vertices are listed by level with respect to the facet $F$.  The $\eta$-vectors
  \[
  \eta^F \ = \ (3,3,1) \ = \eta^G
  \]
  coincide.  The polytope $A$ is a skew bipyramid with apices $e_3$ and $e_1-e_3$ over the pentagon
  $P_5=\conv\{\pm e_1,e_2,\pm(e_1-e_2)\}$.  The latter is smooth Fano and so is $A$ due to
  \autoref{lem:skew-bipyramid}.  The polytope $A$ is the polytope \texttt{F.3D.0002.poly} in the file
  \texttt{fano-v3d.tgz} provided at \cite{smoothreflexive}. The polytope can also be found in the new
  \texttt{polymake} database, \texttt{polydb}, with ID \texttt{F.3D.0002}; see \url{www.polymake.org} for more
  details.
\end{example}

\begin{figure}[ht]
  \centering
		\begin{tikzpicture}[x  = {(2cm,0cm)},
		                    y  = {(1cm,.4cm)},
		                    z  = {(0cm,1.2cm)},
		                    scale = 1,
		                    color = {lightgray}]
		\tikzset{facestyle/.style={fill=green!20, opacity=.8,draw=black,line join=round}}
		\tikzset{specialfacestyle/.style={fill=brown, opacity=.8,draw=black,line join=round}}
		\tikzstyle{vertex}=[circle,minimum size=3pt,inner sep=0pt, fill=black]
		\tikzstyle{every label}=[black]
		
		\draw[white] (1.5,0,0) -- (-1.5,0,0);
		
		  \draw[specialfacestyle] (0,0,1) -- (1,0,0) -- (0,1,0) -- (0,0,1) -- cycle ;
		  \draw[facestyle] (0,0,1) -- (-1,1,0) -- (0,1,0) -- (0,0,1) -- cycle ;
		  \draw[facestyle] (0,0,1) -- (-1,1,0) -- (-1,0,0) -- (0,0,1) -- cycle ;
		
		  \draw[specialfacestyle] (1,0,-1) -- (1,0,0) -- (0,1,0) -- (1,0,-1) -- cycle ;
		  \draw[facestyle] (1,0,-1) -- (-1,1,0) -- (0,1,0) -- (1,0,-1) -- cycle ;
		  \draw[facestyle] (1,0,-1) -- (-1,1,0) -- (-1,0,0) -- (1,0,-1) -- cycle ;

		    \draw (0,1,0) node[label=right:\raisebox{3ex}{$\scriptstyle e_2$}] {};
		    \fill[black] (0,1,0) circle (2pt);
		    \fill[black] (-1,1,0) circle (2pt);
		
		  \draw[facestyle] (0,0,1) -- (1,-1,0) -- (1,0,0) -- (0,0,1) -- cycle ;
		  \draw[facestyle] (0,0,1) -- (-1,0,0) -- (1,-1,0) -- (0,0,1) -- cycle ;
		
		  \draw[facestyle] (1,0,-1) -- (-1,0,0) -- (1,-1,0) -- (1,0,-1) -- cycle ;
		  \draw[facestyle] (1,0,-1) -- (1,-1,0) -- (1,0,0) -- (1,0,-1) -- cycle ;
		
		  \draw (0,0,1) node[label=above:$\scriptstyle e_3$] {};
		  \draw (1,0,-1) node[label=below:$\scriptstyle e_1-e_3$] {};
		  \draw (1,0,0) node[label=right:$\scriptstyle e_1$] {};
		  \fill[black] (0,0,1) circle (2pt);
		  \fill[black] (1,0,-1) circle (2pt);
		  \fill[black] (1,0,0) circle (2pt);
		  \fill[black] (-1,0,0) circle (2pt);
		  \fill[black] (1,-1,0) circle (2pt);
		
		  \fill[black] (1,1,0) circle (2pt);
		  \draw (1,1,0) node[label=right:$\scriptstyle e_1+e_2$] {};

		\end{tikzpicture} 
  \caption{The polytope $A$ from \autoref{example:skewbip-P5}.  Its two special facets are the ones
    containing $e_1$ and $e_2$.}
  \label{fig:skewbip-P5}
\end{figure}
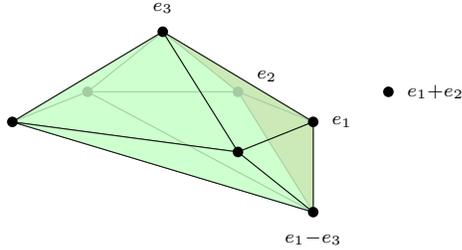

The following example shows that the $\eta$-vector does depend on the choice of the special facet.

\begin{example}\label{example:eta}
  We also consider the smooth Fano $4$-polytope $B$ which is the convex hull of the ten vertices
  \[
  \begin{gathered}
    e_1 \,,\ e_2 \,,\ e_3 \,,\ e_4 \,;\\
    e_1-e_2 \,,\ e_2-e_1 \,;\\
    -e_1 \,,\ -e_2 \,;\\
    e_1 - e_2 - e_3 \,,\ e_2-e_1-e_4 \, .
  \end{gathered}
  \]
  Here the vertex sum vanishes, and so all $24$ facets are special.  Two examples are
  \[
  F \ := \ \conv\{e_1,\, e_2,\, e_3,\, e_4\} \quad \text{and} \quad H \ := \ \conv\{e_2,\, e_3,\,
  e_4,\, e_2-e_1\} \, .
  \]
  Again the vertices are listed by level with respect to the facet $F$.  We have $u_F=\1$ and
  $u_H=-e_2-e_3-e_4=\1-e_1$.  This yields
  \[
  \eta^F \ = \ (4,2,4) \quad \text{and} \quad \eta^H \ = \ (4,3,2,1) \, .
  \]
  The unique vertex at level $-2$ with respect to $H$ is $e_1-e_2-e_3$.  This polytope is \texttt{F.4D.0066.poly} 
  in the file \texttt{fano-v4d.tgz} at \cite{smoothreflexive}. The polytope can also be found in the new 
  \texttt{polymake} database, \texttt{polydb}, with ID \texttt{F.4D.0066}; see \url{www.polymake.org} for more 
  details.
\end{example}

Throughout the paper we will return to these two examples.  Let $F$ be a facet of a simplicial
$d$-dimensional polytope with $F = \conv \{ v_1, \ldots, v_d \}$. To every $v_i \in F$ there is a
unique ridge $R=\conv (\Vert F \setminus \{ v_i \})$ and a unique facet $G \ne F$ such that $G\cap F
= R$. We will call this facet the \emph{neighboring facet} of $F$ with respect to $v_i$, and it will
be denoted by $\neigh(F,v_i)$. There also is a unique vertex $v \in \Vert P$ such that
$\neigh(F,v_i) = \conv( R \cup \{ v \})$. This vertex will be called \emph{opposite vertex} of $v_i$
with respect to $F$ and will be denoted by $\opp(F,v_i)$. To make things a little bit clearer see
\autoref{fig:opposite_vertex}.

\begin{figure}[ht]
  \centering
		\begin{tikzpicture}[scale=1.5]
		  \tikzstyle{edge} = [draw,thick,-,black]
		
		  \coordinate[label=left:$\scriptstyle v$] (v) at (0,0);
		  \coordinate (z) at (1.2,0);
		  \coordinate[label=above:$\scriptstyle w$] (w) at (0,1);
		  \coordinate[label={right:$\scriptstyle \opp(F,v)$}] (ov) at (1.5,1);
		  \coordinate[label={below:$\scriptstyle \opp(F,w)$}] (ow) at (1,-1);
		
		  \draw[edge] (v) -- (w) -- (z) -- (v) -- cycle;
		  \draw[edge] (ov) -- (w) -- (z) -- (ov) -- cycle;
		  \draw[edge] (v) -- (ow) -- (z) -- (v) --cycle;
		  
		  \foreach \point in {v,z,w,ov,ow}
		    \fill[black] (\point) circle (1.3pt);
		  
		  \draw (.3,.3) node {$\scriptstyle F$};
		  \draw (.9,.75) node {$\scriptstyle \neigh(F,v)$};
		  \draw (.7,-.25) node {$\scriptstyle \neigh(F,w)$};
		\end{tikzpicture} 
  \caption{Neighboring facets and opposite vertices}
  \label{fig:opposite_vertex}
\end{figure}
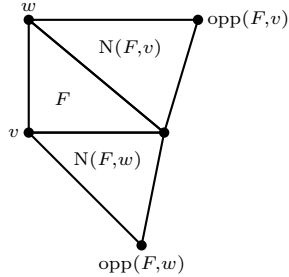

\begin{example}
  Let $B$ be the $4$-polytope from \autoref{example:eta}. The two facets $F$ and $H$ are adjacent
  with $H=\neigh(F,e_1)$ and $\opp(F,e_1)=e_2-e_1$.
\end{example}

\subsection{Coordinates and Lattice Bases}

Throughout let $P\subset\RR^d$ be a simplicial and reflexive $d$-polytope for $d\ge 2$.  Let $v$ be
a vertex of $P$ and let $F$ be a facet containing $v$.  The vertices of $F$ form a basis of $\RR^n$,
and hence there is a unique vector $u_{F,v}$, the \emph{vertex normal} of $v$ with respect to $F$,
satisfying
\[
\langle u_{F,v}, v\rangle = 1\,, \qquad \text{and} \quad \langle u_{F,v},w\rangle = 0 \text{ for all }
w\in\Vert{F}\setminus\{v\} \, .
\]
That is to say, the set $\smallSetOf{u_{F,v}}{v\in\Vert{F}}$ is the basis dual to $\Vert{F}$.  For
the reader's benefit we explicitly list a few known facts which will be useful in our proofs below.
Notice that here we do not assume that the vertices of $F$ form a \emph{lattice} basis.

\begin{lemma}[{\Obro~\cite[Lem.~1 and~2]{Obro08}}]\label{lemma:2}
  Let $F$ be a facet of $P$, $v$ a vertex of $F$, and $G:=\neigh(F,v)$. If $x\in P$ then
  \[
  \langle u_G,x \rangle \ = \ \langle u_F, x\rangle + \left( \langle u_G,v \rangle -1 \right) \langle u_{F,v} ,x \rangle
  \]
  Moreover $\langle u_{F}, x \rangle -1 \le \langle u_{F,v}, x \rangle$. In case of
  equality we have $x = \opp(F,v)$.
\end{lemma}

\begin{example}\label{lem:coordinates_at_minus1}
  Let $P$ be a \str $d$-polytope such that the standard basis vectors span a facet $F=\conv\{ e_1,
  e_2,\ldots, e_d \}$.  Further, let $x \in V(F,-1)$ and $z=\opp(F,e_i)$ for some $i\in[d]$, with
  $x\ne z$. The previous lemma shows that
  \[
  x_i \ = \ \langle u_{F,e_i}, x \rangle  \ = \ \frac{\langle u_{\neigh(F,e_i)}, x \rangle +1}{(\langle \1, z \rangle -1 )}
  \]
  Moreover, this gives
  \[
  \langle u_{\neigh(F,e_i)}, x \rangle \ < \ 1 -2 \langle \1,z \rangle \, .
  \]
\end{example}

We say that two lattice points $v,w\in \partial P\cap \ZZ^n$ are \emph{distant} if they are not
contained in a common face. We have the following facts about distant lattice points.

\begin{lemma}[{Nill~\cite[Lem.~4.1]{1067.14052}}]\label{lemma:nill_sumDistant}
  Consider two lattice points $v,w$ in the boundary of $P$ such that $v\ne -w$.  Then
  $v+w\in\partial P\cap \ZZ^n$ if and only if $v$ and $w$ are distant.
\end{lemma}

\begin{example}
  The polytope A of \autoref{example:skewbip-P5} has four pairs of distant vertices: $(e_1,
  e_2-e_1)$, $(e_1,-e_2)$, $(e_2,e_1-e_2)$, and $(e_3,e_1-e_3)$.
\end{example}

\begin{proposition}[{Nill~\cite[Lem.~5.5]{1067.14052}}]\label{prop:nill:5.5}
  Let $F$ be a facet of $P$, and let $x\in\ZZ^d$ be any lattice point in $\partial P\cap H(F,0)$.
  Then
  \begin{enumerate}
  \item\label{prop:nill:5.5:1} the point $x$ is contained in a facet adjacent to $F$,
  \item\label{prop:nill:5.5:2} for each vertex $v$ of $F$ we have $x\ne\opp(F,v)$ if and only if
    $\langle u_{F,v}, x \rangle \ge 0$,
  \item\label{prop:nill:5.5:3} if there is a vertex $v$ of $F$ such that $x=\opp(F,v)$ and
    $x\ne\opp(F,w)$ for any other vertex $w\ne v$ of $F$, then $v$ and $x$ are distant.
  \end{enumerate}
\end{proposition}

This technical but powerful result has the following consequence. Remember that $P$ is a simplicial polytope.

\begin{corollary}\label{cor:opposite_at_0}
  Let $F\subseteq P$ be a facet and $x\in V(F,0)$. Then $x$ is opposite to some vertex of $F$.
\end{corollary}

For the sake of brevity we call a vertex $v$ of the facet $F$ \emph{good} if $\opp(F,v)$ is
contained in $V(F,0)$ and the equality $\langle u_{F,v},\opp(F,v) \rangle=-1$ holds.

\begin{example}
  All four vertices of the facet $F$ of the polytope $B$ from \autoref{example:eta} are good: for
  instance, $u_{F,e_1}=e_1$ and $\opp(F,e_1)=e_2-e_1$ such that $\langle u_{F,e_1},\opp(F,e_1)
  \rangle=-1$.
\end{example}

\begin{lemma}[{Nill~\cite[Lem.~5.5]{1067.14052}}]\label{lem:facet_basis}
  If the facet $F$ of $P$ contains at least $d-1$ pairwise distinct good vertices, then the vertices
  of $F$ form a lattice basis.
\end{lemma}

A recurring theme in our paper is that the \str polytopes within the
class we consider turn out to be smooth.  In this the preceding observation is a key ingredient.

\subsection{Characterizing Vertices in \texorpdfstring{$V(F,0)$}{V(F,0)}}
\phantomsection\label{subsec:HF0} Throughout this section we assume that $P\subset\RR^d$ is a
$d$-dimensional \str polytope with a fixed facet $F$.  The purpose of this section is to investigate
the situation where the hyperplane $H(F,0)$ contains ``many'' vertices.  In particular, if there are
at least $d-1$ good vertices in $V(F,0)$ then we can apply \autoref{lem:facet_basis}.  In this case,
up to a unimodular transformation, we may assume that the standard basis vectors of $\RR^d$ coincide
with the vertices of $F$.

It will be convenient to phrase some of the subsequent results on the vertices in $V(F,0)$ in terms
of the function
\begin{equation}\label{eq:phi}
\phi \,:\, \Vert{F}\to\Vert{F}\cup\{\0\} \,,\; v\mapsto\begin{cases} w & \text{if }
  w=\opp(F,v)+v\in\Vert{F} \\ \0 & \text{otherwise\,.}\end{cases}
\end{equation}
Occasionally, we additionally let $\phi(\0)=\0$ in which case $\0$ becomes the only fixed point.
Clearly, like the $\eta$-vector also the function $\phi$ depends on the choice of the facet $F$.  If
we want to express this dependence we write $\phi^F$ instead.  Vertices that satisfy $\phi(v)\ne\0$
are necessarily good, since then $\opp(F,v)=\phi(v)-v$ and
\[
\langle u_{F,v},\opp(F,v)\rangle \ = \ \langle u_{F,v},\phi(v)\rangle-\langle u_{F,v},v\rangle \ =
\ 0-1 \ =\ -1 \, .
\]
However the converse is not true.  It is possible that $v$ is a good vertex and satisfies $\phi(v) = 0$.
It means that the opposite vertex of $v$ is not of the form $\phi(v) - v$ but it could still be in $V(F,0)$.

\begin{example}
  The function $\phi^F$ for the facet $F$ of the $3$-polytope $A$ from \autoref{example:skewbip-P5}
  reads as follows: $\phi^F(e_1)=e_2$ and $\phi^F(e_2)=\phi^F(e_3)=e_1$. This implies that $e_1$, $e_2$
  and $e_3$ are good, but as said above not every good vertex has $\phi(v)\ne 0$. To see this look
  at the del Pezzo polytope $\HSBC{4}$ in dimension $4$. The facet $H=\conv\{ e_1,e_2,-e_3,-e_4 \}$
  has outer facet normal vector $u_H = e_1+e_2-e_3-e_4$. In this situation all vertices of $H$ are good
  as $\opp(H,e_1) = \opp(H,e_2) = -\1$ and $\opp(H,-e_3) = \opp(H,-e_4) = \1$ and $\pm \1 \in V(H,0)$.
  But we get $\phi^H(v)=0$ for every vertex $v$ of $H$ as $\pm \1$ cannot be expressed as a sum of only
  two vertices in $H$.
\end{example}

Notice that $\phi(v)\ne \0$ implies that the vertex $\opp(F,v)=\phi(v)-v$ is contained in the set
$V(F,0)$. The following partial converse slightly strengthens \cite[Lem.~6]{Obro08}, but the proof
is essentially the same.

\begin{lemma}\label{lem:one_vertex_at_0}
  Let $v$ be a vertex of $F$ such that $\opp(F,v) \in V(F,0)$.  If for every vertex $w\in\Vert{F}$
  other than $v$ we have that $\opp(F,w) \ne \opp(F,v)$ then $\phi(v)\ne \0$.
\end{lemma}

\begin{proof}
  The vertices of $F$ form a basis of $\RR^d$, albeit not necessarily a lattice basis.  Using
  \autoref{prop:nill:5.5}.(\ref{prop:nill:5.5:2}) and our assumption we can write
  \[
  v' \ := \ \opp(F,v) \ = \ -\alpha v + \sum_{w \in W} \beta_w w
  \]
  for some set $W\subset\Vert{F}\setminus\{v\}$ and $\alpha>0$ as well as $\beta_w > 0$ for all
  $w\in I$. Since $v' \in V(F,0)$ we have $\sum_{w\in W}\beta_w = \alpha$, and the set $W$ is not
  empty.  If there were a facet $G$ containing both $v$ and $v'$ then
  \[
  1+\alpha \ = \ \langle u_G, v' + \alpha v \rangle \ = \ \langle u_G, \sum_{w\in W} \beta_w w
  \rangle \ \le \ \sum_{w\in W}\beta_w \ = \ \alpha \, ,
  \]
  which is a contradiction. So $v$ and $v'$ are distant, and by \autoref{lemma:nill_sumDistant}
  $v+v'\in\partial P$.  The polytope $P$ is terminal by assumption, so $v+v'$ is a vertex.  As
  $v'\in V(F,0)$ we have $v+v' \in V(F,1)$.  So there must be a vertex $w$ of $F$ with $w =
  v+v'$. Hence, $v'=w-v$, or, equivalently, $\phi(v) = w$.
\end{proof}

The following is now a direct consequence of Lemmas~\ref{lem:one_vertex_at_0}, \ref{lem:facet_basis} and
\autoref{cor:opposite_at_0}.

\begin{lemma}[{\Obro~\cite[Lem.~6]{Obro08}}]\label{lem:d_vertices_at_0}
  If $\eta_0=d$ then $ V(F,0) = \SetOf{ \phi(v)-v }{ v\in\Vert{F} }$.  In particular, the vertices
  of $F$ form a lattice basis.
\end{lemma}

We can investigate a situation similar to \autoref{lem:one_vertex_at_0} more closely.  It indicates
why \str polytopes with many vertices are prone to be smooth.

\begin{proposition}\label{lem:eta_0_d-1_smooth}
  Let $P$ be a $d$-dimensional \str polytope such that $F$ is a special
  facet of $P$ with $\eta^F_0 \ge d-1$.  Then the vertices of $F$ form a lattice basis. Further, for
  at least $d-2$ vertices of $F$ we have $\phi(v)\ne \0$.
\end{proposition}

\begin{proof}
  The case $\eta^F_0=d$ is handled in \autoref{lem:d_vertices_at_0}, so we assume that $\eta_0=d-1$.
  Since every element in $V(F,0)$ must be opposite to some vertex of $F$ we see that the conditions
  of \autoref{lem:one_vertex_at_0} are satisfied for at least $d-2$ vertices of $F$.  So we are left
  with two cases:
  \begin{itemize}
  \item there exists exactly one vertex $v$ of $F$ with $\opp(F,v) \not\in V(F,0)$, or
  \item there are exactly two vertices $v$ and $w$ of $F$ with $\opp(F,v)=\opp(F,w) \in V(F,0)$.
  \end{itemize}
  In the first case we have $d-1$ good vertices. This makes the vertices of $F$ a lattice basis
  according to \autoref{lem:facet_basis}.

  In the second case we look at the vertex $z := \opp(F,v)=\opp(F,w) \in V(F,0)$.  We express
  $z=-z_vv-z_ww+\sum_{x\in\Vert{F}\setminus\{v,w\}} z_x x$ in the coordinates given by the basis $\Vert{F}$ of
  $\RR^d$.  We have $z_v=-\langle u_{F,v}, z \rangle > 0$ and $z_w=-\langle u_{F,w}, z \rangle > 0$.  The
  remaining $d-2$ vertices of $F$ are good. Hence \autoref{prop:nill:5.5}.(\ref{prop:nill:5.5:2}) forces
  $z_x=\langle u_{F,x}, z \rangle\ge 0$ for all $x\ne v,w$.

  Suppose that $z_x$ is not integral for some $x\in\Vert{F}\setminus\{v,w\}$, and let
  $G:=\neigh(F,x)$. Note that $\langle u_G, x \rangle =\langle u_G, \phi(x) \rangle - \langle u_G,
  \phi(x) -x \rangle = 1-1 = 0$, since $\phi(x)$ and $\phi(x)-x$ are vertices of $G$.  By
  \autoref{lemma:2}
  \[
  \langle u_G,z \rangle \ = \ \langle u_F, z\rangle + \left( \langle u_G,x \rangle -1 \right) z_x \ = \ - z_x\,,
  \]
  so $\langle u_G,z \rangle$ would not be integral, a contradiction.  Hence
  \[
  z_v + z_w \ = \ \sum_{y\in\Vert F \setminus \{ v,w \}} z_y
  \]
  is an integer.  The second part of \autoref{lemma:2} also implies that $z_v, z_w \le 1$. Hence,
  $z_v+z_w\in \{1,2\}$. If $z_v+z_w=2$, then $z_v = z_w = 1$, and the vertices $v$ and $w$ are
  good. Hence the vertices of $F$ form a lattice basis. If $z_v + z_w = 1$, then $z_vv+z_ww$ is a
  proper convex combination of $v$ and $w$. We can write this as
  \begin{align*}
    z_vv+z_ww\ = \ -z\ +\ \sum_{y\in\Vert F \setminus \{ v,w \}} z_y y
  \end{align*}
  Since both summands on the right hand side are integral this implies that $z_v v + z_w w\in \ZZ^d$.
  This contradicts the terminality of $P$.
\end{proof}

The next result helps to identify facets.

\begin{lemma}\label{lem:2times_neighboring}
  Let $v$ and $w$ be distinct vertices of $F$ such that $\phi(v)\ne \0$ and
  $\phi(w)\not\in\{\0,v\}$. Then the $(d{-}1)$-simplex
  \[
  \conv\biggl( \bigl(\Vert{F}\setminus\{v,w\}\bigr) \cup \{ \phi(v)-v,\phi(w)-w \} \biggr)
  \]
  is a facet of $P$.
\end{lemma}

\begin{proof}
  As $\phi(v)\ne \0$ we know that $v':=\phi(v)-v=\opp(F,v)$, and this means that $\conv
  ((\Vert{F}\setminus\{v\})\cup\{v'\})=\neigh(F,v)$ is a facet.  For $w':=\phi(w)-w=\opp(F,w)$ the
  assumption $\phi(w)\not\in\{\0,v\}$ yields
  \begin{align*}
    \langle u_{\neigh(F,v)},w'\rangle \ &= \ \langle u_{\neigh(F,v)},\phi(w)\rangle - \langle
    u_{\neigh(F,v)},w \rangle \ = \ 1-1 \ = \ 0\,
  \end{align*}
  as both $\phi(w)$ and $w$ are vertices of $\neigh(F,v)$, and because $P$ is reflexive the scalar product of
  those vertices with the facet normal vector evaluates to $1$. Furthermore we get
  \begin{align*}
     \langle u_{\neigh(F,v),w},w'\rangle \ &= \ \langle u_{\neigh(F,v),w},\phi(w)\rangle - \langle
    u_{\neigh(F,v),w},w \rangle \ = \ 0-1 \ = \ -1\,.
  \end{align*}
  Now \autoref{lemma:2} says that $\opp(\neigh(F,v),w)=w'$, and hence,
  $\Vert{\neigh(F,v)}\setminus\{ w \}) \cup \{ w' \}$ is the set of vertices of the facet
  $\neigh(\neigh(F,v),w)$ of $P$.  This is the claim.
\end{proof}

Next we want to extract more information about the function $\phi$ by taking into account vertices
at level $-1$ and below.  If $d$ is even and $P$ has at least $3d-1$ vertices with $\eta_0=d$ then
it follows from \autoref{thm:3d-1} that $\phi$ is an involutory permutation of the vertices of $F$;
that is, we have $\phi(\phi(v))=v$ for all $v\in\Vert{F}$.  This can be generalized as follows.
Recall that $\phi(v)=\0$ means that the vertex $\opp(F,v)$ is not of the form $w-v$ for any vertex
$w\in\Vert{F}$.

\begin{lemma}\label{lem:phi}
  Let $v$ be a vertex of $F$ satisfying $\phi(v)\ne \0$. If $-v$ is a vertex of $P$ then $\phi(\phi(v)) \in \{ \0,v \}$.
\end{lemma}

 \begin{proof}
   Aiming at a contradiction suppose that $\phi(\phi(v)) =x$ for some
   vertex $x\ne v$ of $F$.  Letting $w=\phi(v)$ the three vertices $v,w,x$ are pairwise distinct.
   Now \autoref{lem:2times_neighboring} gives us that
   \begin{equation}\label{eq:G}
     G \ = \ \conv\biggl( \bigl(\Vert{F}\setminus\{v,w\}\bigr) \cup \{ w-v,x-w \} \biggr)
   \end{equation}
   is a facet. In particular, $G$ contains the three vertices $w-v$, $x-w$, and~$x$.  Hence the
   equation
   \[
   1 \ = \ 1+1-1 \ = \ \langle u_G, w-v \rangle + \langle u_G, x-w \rangle - \langle u_G, x \rangle \ = \ \langle u_G, -v \rangle
   \]
   shows that $-v$ is contained in $G$, too.  However $-v$ is none of the $d$ vertices listed in
   \eqref{eq:G}.  This is the contradiction desired, as each face of $P$ is a simplex.
 \end{proof}

 \begin{example}
   The polytope $A$ of \autoref{example:skewbip-P5} contains the vertex $-e_1$, which implies that
   $\phi(\phi(e_1))=\phi(e_2)=e_1$.
 \end{example}

The following lemma says how good vertices at level zero restrict the vertices at level $-1$.

\begin{lemma}[{\Obro~\cite[Lem.~5]{Obro08}}]\label{lem:phi:notwonegative}
  Let $v$ and $w\ne v$ be good vertices of $F$ with $\opp(F,v) \ne \opp(F,w)$. Then there is no vertex $x\in
  V(F,-1)$ such that $\langle u_{F,v}, x\rangle = \langle u_{F,w}, x\rangle = -1$.
\end{lemma}

The subsequent lemma is an important clue in the proof of \Obro's classification \autoref{thm:3d-1}.

\begin{lemma}[{\Obro~\cite[Lem.~7]{Obro08}}]\label{lem:vertices_at_minus1}
  If $\eta_0=d$ then all vertices in $V(F,-1)$ are of the form $-v$ for some vertex $v\in\Vert{F}$.
\end{lemma}

We summarize all considerations in the following Proposition.  To simplify the notation we let
$\opp(F):=\smallSetOf{\opp(F,v)}{v\in\Vert{F}}$ for any facet $F$ of the polytope~$P$.  Further, we
we abbreviate ``pairwise distinct'' as ``p.d.''.

\begin{proposition}\label{prop:classificationOf0and-1}
  Let $P$ be a $d$-dimensional \str polytope such that $F$ is a special facet.
  \begin{enumerate}
  \item \label{prop:classificationOf0and-1:1} If $\eta^F_0=d$, then, up to lattice equivalence,
    $F=\conv\{e_1,e_2, \ldots, e_d\}$, and
    \begin{align*}
      V(F,0)\ &=\ \{\phi(e_1)-e_1,\, \phi(e_2)-e_2,\, \ldots,\, \phi(e_d)-e_d\}\\
      V(F,-1)\ &\subseteq\ \{-e_1,\, -e_2,\, \ldots,\, -e_d\}\,.
    \end{align*}
  \item \label{prop:classificationOf0and-1:2} If $\eta^F_0=d-1$ and $\opp(F) = V(F,0)$, then,
    up to lattice equivalence, $F=\conv\{e_1, e_2, \ldots, e_d\}$, and
    \begin{align*}
      V(F,0)\ &=\ \{-e_1-e_2+e_a+e_b, \, \phi(e_3)-e_3,\,  \ldots,\, \phi(e_d)-e_d\}\\
      V(F,-1)\ &\subseteq\ \{-e_1,\, -e_2,\, \ldots,\, -e_d\} \;\cup\; \SetOf{-e_1-e_2 + e_s}{s\in [d]}
    \end{align*}
    for $a, b \in [d] \setminus\{ 1,2 \}$ not necessarily distinct.
  \item \label{prop:classificationOf0and-1:3} If $\eta^F_0=d-1$ and $\opp(F)\ne
    V(F,0)$, then, up to lattice equivalence, $F=\conv\{e_1,e_2, \ldots, e_d\}$,
    and
    \begin{align*}
      V(F,0)\ &=\ \{\phi(e_2)-e_2,\, \phi(e_3)-e_3,\, \ldots,\, \phi(e_d)-e_d\}\\
      V(F,-1)\ &\subseteq\ \{-e_1,\, -e_2,\, \ldots,\, -e_d\} \;\cup\; \SetOf{-2e_1-e_r +e_s+e_t}{r,s,t \in [d]
        \text{ p.d.},\, r\ne 1}\,. 
    \end{align*}
  \end{enumerate}
\end{proposition}

\begin{proof}
  Assume first that $\eta^F_0=d$. Then \autoref{lem:d_vertices_at_0} implies that the vertices of
  $F$ form a lattice basis, so up to a lattice transformation we can assume that $F=\conv\{e_1,
  e_2,\ldots, e_d\}$. The same lemma gives $V(F,0) = \{\phi(e_1)-e_1, \ldots, \phi(e_d)-e_d\}$.  By
  \autoref{lem:vertices_at_minus1} now $V(F,-1)\subseteq\{-e_1, -e_2, \ldots, -e_d\}$. This proves
  \eqref{prop:classificationOf0and-1:1}.

  From now on we assume that $\eta^F_0=d-1$. We look first at the case $\opp(F) = V(F,0)$.  By
  \autoref{lem:eta_0_d-1_smooth} we know that the vertices of $F$ form a lattice basis, so, up to a
  lattice transformation, $F=\conv\{e_1,e_2, \ldots, e_d\}$. Further, the function $\phi$ does not
  vanish for at least $d-2$ vertices in $F$.  Since any vertex in $V(F,0)$ is opposite to a vertex
  of $F$ we know that there is a vertex $x\in V(F,0)$ such that $x:=\opp(F,v)=\opp(F,w)$ for
  precisely two vertices $v,w$ of $F$. Hence, up to relabeling,
  \begin{align*}
    V(F,0)\ =\ \{\phi(e_3)-e_3, \ldots, \phi(e_d)-e_d\} \;\cup\; \{x\}
  \end{align*}
  Let $x=(x_1, \dots, x_d)$. By \autoref{lemma:2} and
  \autoref{prop:nill:5.5}.(\ref{prop:nill:5.5:2}) we have $-1\le x_1, x_2<0$ and $x_j\ge 0$ for
  $j\ge 3$.  $x\in\ZZ^d$ implies $x_1=x_2=-1$. And since $\sum_{i=1}^d x_i = 0$ we get that there
  exist not necessarily distinct indices $a,b \in [d]\setminus\{ 1,2 \}$ with $x=-e_1-e_2
  +e_a+e_b$. Now we take a look at an arbitrary $y\in V(F,-1)$. Because of
  \autoref{lem:phi:notwonegative} we know that for two distinct indices $i$ and $j$ it is not
  possible that $y_i = y_j = -1$ except for $i=1$ and $j=2$ (or vice versa). And since every vertex
  of $F$ is good \autoref{lemma:2} gives us that $y_i \ge -1$ for all $i \in [d]$.  Putting this
  together with the fact that $\sum_{i=1}^d y_i = -1$ leaves us with either $y \in \{ -e_1,
  -e_2,\ldots ,-e_d \}$ or $y=-e_1-e_2+e_s$ for some $s\in [d]$. This proves
  \eqref{prop:classificationOf0and-1:2}.

  So finally assume that there is some vertex $v$ in $F$ with $\opp(F,v) \not\in V(F,0)$. Up to
  relabeling, $v=e_1$. Then
  \[
  V(F,0)\ =\ \{\phi(e_2)-e_2, \ldots, \phi(e_d)-e_d\}\,.
  \]
  because of the same argument as above. That is, every vertex in $V(F,0)$ must be opposite to a
  unique vertex in $F$.  Hence, by \autoref{lem:one_vertex_at_0} we know the entire set $V(F,0)$.
  Now we look at $y\in V(F,-1)$.  Because of \autoref{lem:phi:notwonegative} we know that for two
  distinct indices $i,j \in [d] \setminus \{ 1 \}$ it is not possible that $y_i = y_j = -1$, since
  all vertices are good except $e_1$. Now let $G = \neigh(F,e_1)$. \autoref{lemma:2} now implies
  several things. Firstly, we get $y_i \ge -1$ for $i \in [d] \setminus \{ 1 \}$. Together with
  $\langle u_G, y \rangle \le 1$ we get
  \[
  2 \ \ge \ (\langle u_G, e_1 \rangle -1) x_1 \quad \text{if and only if} \quad x_1 \ \ge \
  \frac{2}{\langle u_G,e_1 \rangle -1}\, .
  \]
  This implies $x_1  \ge -2$ as $\langle u_G, e_1 \rangle \le 0$ holds.
  So we are left with $y\in \{ -e_1,-e_2,\ldots, -e_d \}$ or $y=-2e_1-e_r +e_s+e_t$ for some
  $r,s,t$. Now we want to show that $r,s,t$ are pairwise distinct and $r\ne 1$. We look at all the different cases.
  \begin{itemize}
    \item $r=s=t=1$ would imply $y = -e_1$, a case which is already covered, or
    \item $s=t$ would imply $y = -2e_1 -e_r + 2e_s$, which is impossible since $P$ is terminal and $-e_1+e_s \in \conv\{ y, e_r \}$, or
    \item $r=t$ (the case $r=s$ is similar) would imply $y = -2e_1 + e_s$, which is impossible since $P$ is terminal and $-e_1+e_s \in \conv\{ y, e_s \}$.
  \end{itemize}
  So $r,s,t$ are pairwise distinct. From $x_1 \ge -2$ we get that $r\ne 1$. This proves \eqref{prop:classificationOf0and-1:3}.
\end{proof}

\section{Smooth Fano  \texorpdfstring{$d$}{d}-polytopes with \texorpdfstring{$3d-2$}{3d-2} vertices}

Throughout this section we assume that $P\subset\RR^d$ is a $d$-dimensional \str polytope with
precisely $3d-2$ vertices.  It is a consequence of our main result, which we will prove below, that
all such polytopes turn out to be smooth Fano.  Let $F$ be a special facet $F$ of $P$, that is, the
vertex sum $v_P$ is contained in $\pos(F)$.  We explore the possible shapes of the vector
$\eta=\eta^F$.  As in \Obro~\cite[\S5]{Obro08} it is useful to expand the expression $\langle
u_F,v_P \rangle$.  Since $\eta_1=d$ we obtain
\begin{equation}\label{eq:sum}
  0 \ \le \ \langle u_F,v_P \rangle \ = \ d + \sum_{k\le -1} k \cdot \eta_k \, .
\end{equation}
Hence, in particular, $\eta_{-1}+\eta_{-2}+\dots\le d$.  By \autoref{cor:opposite_at_0} any vertex
at level zero with respect to $F$ is opposite to a vertex in $F$, and the total number of vertices
is $3d-2$.  So we have
\begin{align}
  d-2 \ &\le \ \eta_0 \ \le d \,,\label{eq:H0}\\
  d-2 \ &\le \ \sum_{k\le -1} \eta_k \ \le \ d \, .\label{eq:Hbelow0}
\end{align}
The linear restrictions \eqref{eq:sum}, \eqref{eq:H0}, and \eqref{eq:Hbelow0} now leave us with
finitely many choices for $\eta$.  The resulting six admissible $\eta$-vectors are displayed in
Table~\ref{tab:eta}.  In particular, $\eta_k=0$ for $k\le-4$.  Notice that in
\autoref{prop:novertexat-3} below we will show that the $\eta$-vector $(d,d,d-3,0,1)$ for $v_P=\0$
does not occur. The case distinction by the eccentricity $\ecc(P)$ of the polytope $P$ is the guiding principle for the proof of our main result.

\begin{table}[tb]
  \renewcommand{\arraystretch}{0.9}
  \begin{tabular*}{.75\linewidth}{@{\extracolsep{\fill}}rccccccc@{}}\toprule
      $\ecc(P)$   & $2$   & $1$   & $1$   & $0$   & $0$   & $0$   & $0$   \\
     \midrule
      $\eta_1$    & $d$   & $d$   & $d$   & $d$   & $d$   & $d$   & $d$   \\
      $\eta_0$    & $d$   & $d$   & $d-1$ & $d$   & $d$   & $d-1$ & $d-2$ \\
      $\eta_{-1}$ & $d-2$ & $d-3$ & $d-1$ & $d-3$ & $d-4$ & $d-2$ & $d$   \\
      $\eta_{-2}$ & $0$   & $1$   & $0$   & $0$   & $2$   & $1$   & $0$   \\
      $\eta_{-3}$ & $0$   & $0$   & $0$   & $1$   & $0$   & $0$   & $0$   \\
     \bottomrule\\
  \end{tabular*}

  \caption{List of possible $\eta$-vectors of \str $d$-polytopes with $3d-2$ vertices, where $\ecc(P)$
    denotes the eccentricity of $P$.}
  \label{tab:eta}
\end{table}

\begin{example}
  Let us consider the planar case $d=2$ .  The classification in \autoref{fig:2dmin_Fano} shows
  that there are two \str polygons with $3\cdot 2-2=4$ vertices.  The
  vertex sum of $P_{4a}$ is zero, and the vertex sum of $P_{4b}$ equals $e_1$.  Taking
  $\conv\{e_1,e_2\}$ as a special facet in both cases the $\eta$-vector for that facet of $P_{4a}$
  reads $(2,0,2)$, corresponding to the last column of Table~\ref{tab:eta}.  For $P_{4b}$ the
  $\eta$-vector of that facet reads $(2,1,1)$, corresponding to the third column.
\end{example}

Let $P$ be a \str $d$-polytope with $3d-1$ vertices, and let $Q$ be an \str $e$-polytope with
$3e-1$ vertices, as in \Obro's classification \autoref{thm:3d-1}.  Then $P\oplus Q$ is a \str
polytope of dimension $d+e$ with $3(d+e)-2$ vertices.  Clearly, the vertex sum is $v_{P\oplus
  Q}=v_P+v_Q$.  If $F$ is a special facet for $P$ and $G$ is a special facet for $Q$ then the joint
convex hull $\conv(F\cup G)$ in $\RR^{d+e}$ is a special facet for $P\oplus Q$.  A direct
computation gives
\[
\eta^{\conv(F\cup G)}_\ell \ = \ \eta^F_\ell+\eta^G_\ell \, .
\]
It turns out that the only polytopes in dimension $\ge 4$ in our classification which are not of
this type are $\HSBC{4}$ and its direct sums with copies of the del Pezzo hexagon $\HSBC{2}\cong
P_6$.

\subsection{Polytopes of Eccentricity \texorpdfstring{$2$}{2}}

We consider the situation where $v_P\in H(F,2)$ for any special facet $F$.  According to
\autoref{tab:eta} we have $\eta=(d,d,d-2)$.

\begin{proposition}\label{prop:level2}
  Let $P$ be a $d$-dimensional \str $d$-polytope with exactly $3d-2$ vertices, where $d\ge 4$.

  If $\ecc(P)=2$ then the polytope $P$ is lattice equivalent to either a
  skew bipyramid over a $(d{-}1)$-dimensional smooth Fano polytope with $3d-4$ vertices, or to the
  direct sum $P_5^{\oplus 2} \oplus P_6^{\oplus \frac{d}{2}-2}$.  In the latter case $d$ is
  necessarily even.
\end{proposition}
Notice that $3d-4=(3d-2)-2=3(d-1)-1$, and therefore the possible bases of the bipyramids are
explicitly known from \autoref{thm:3d-1}.
\begin{proof}
  Let $F$ be a special facet and $\phi:=\phi^F$.  Since $\eta_0=d$
  \autoref{prop:classificationOf0and-1} shows that $\Vert F$ is a lattice basis, and up to a
  unimodular transformation we may assume that $F=\conv\{e_1,e_2,\dots,e_d\}$ and
  \begin{align*}
    V(F,0) \ &= \ \SetOf{\phi(e_i)-e_i}{i\in[d]} \\
    V(F,-1) \ &= \ \{-e_3,\, -e_4,\, \ldots,\, -e_d\} \, .
  \end{align*}
  From \autoref{lem:phi} we know that $\phi(\phi(e_i))=e_i$ for all $i>2$.  We distinguish whether
  or not $e_1$ or $e_2$ occur in the image of $\phi$.

  Assume first that $|\{e_1,e_2\}\cap\image\phi|\le 1$.  Then, up to symmetry,
  $e_1\not\in\image\phi$.  Browsing our vertex lists above we see that the first coordinate is $0$
  for all vertices except $e_1$ and $\phi(e_1)-e_1$.  So $P$ is a skew bipyramid with apices $e_1$
  and $\phi(e_1)-e_1$ over some $d-1$ dimensional polytope $Q$.  By \autoref{lem:skew-bipyramid} it
  is a \str polytope with $3d-4$ vertices, and a smooth Fano polytope by \autoref{thm:3d-1}.

  It remains to consider the case $\{e_1, e_2\}\subset\image\phi$. If $\phi(e_r)=e_1$ for some $r\ne
  1,2$, then $\phi(\phi(e_r))=e_r$ by \autoref{lem:phi}. Thus, $\phi(e_1)=e_r$ and similarly for
  $e_2$. So either $\phi(e_1)=e_2$ and $\phi(e_2)=e_1$, or there are distinct $r, s\not\in\{ 1,2 \}$
  with $\phi(e_r)=e_1$ and $\phi(e_s)=e_2$.

  Assume first that $\phi(e_1)=e_2$ and $\phi(e_2)=e_1$.  Then $\phi(e_i)\not\in\{e_1,e_2\}$ for all
  $i\ge 3$. The two-dimensional subspace spanned by $e_1$ and $e_2$ contains the four vertices $e_1,
  e_2, \pm(e_1-e_2)$ of $P$, while the remaining $3d-6=3(d-2)$ vertices of $P$ are contained in the
  subspace spanned by $e_3,e_4,\dots,e_d$.  Hence, $P$ decomposes into a direct sum of a quadrangle
  and a $(d{-}2)$-dimensional polytope.  But the quadrangle $\conv\{e_1, e_2, \pm(e_1-e_2)\}$ has
  $\0$ in the boundary, so $P$ is not reflexive.

  Hence, there exist distinct $r, s\not\in\{ 1,2 \}$ with $\phi(e_r)=e_1$, $\phi(e_s)=e_2$,
  $\phi(e_1)=e_r$, and $\phi(e_2)=e_s$.  We conclude that $\phi$ is an involutory fixed-point free
  bijection on the set $\Vert{F}$.  Therefore, $d$ must be even.  Each linear subspace $\lin\{ e_i,
  \phi(e_i) \}$ intersects $P$ in a subspace containing $5$ or $6$ vertices, and all other vertices
  are in a complementary subspace. Thus, $P$ splits into
  \begin{align*}
    P\ =\ P_5^{\oplus 2}\oplus P_6^{\oplus (\frac{d}{2} - 2)}\,,
  \end{align*}
  where the two copies of $P_5$ are contained in $\lin\{ e_1, e_r \}$ and $\lin \{e_2, e_s \}$.
\end{proof}

\begin{example}
  The $3$-polytope $A$ from \autoref{example:skewbip-P5} satisfies the conditions stated in
  \autoref{prop:level2}.  The dimension is odd, and $A$ is a skew bipyramid over the lattice
  pentagon~$P_5$.
\end{example}

\subsection{Polytopes with Eccentricity \texorpdfstring{$1$}{1}}

If the vertex sum $v_P$ lies on level one with respect to the special facet $F$, then $v_P$ is a
vertex of $F$, as $P$ is terminal.  By \autoref{tab:eta} we either have $\eta=(d,d,d-3,1)$ or
$\eta=(d,d-1,d-1)$.  The first case is the easier one; so we will treat it right away.

\begin{proposition}\label{prop:noLevelm2}
  Let $P$ be a $d$-dimensional \str polytope with exactly $3d-2$ vertices,
  where $d\ge 4$.

  If the vertex sum $v_P$ is a vertex on a special facet $F$ with $\eta^F=(d,d,d-3,1)$, then $P$ is
  lattice equivalent to a skew bipyramid over a $(d{-}1)$-dimensional smooth Fano polytope with
  $3d-4$ vertices.
\end{proposition}
Again we have $3d-4=3(d-1)-1$, and the possible bases of the bipyramids are classified in
\autoref{thm:3d-1}.
\begin{proof}
  Let $\phi:=\phi^F$. \autoref{prop:classificationOf0and-1}\eqref{prop:classificationOf0and-1:1}
  shows that $\Vert F$ is a lattice basis, and we may assume that $F=\conv\{e_1, e_2, \ldots, e_d\}$
  as well as $v_P=e_1$.  Furthermore,
  \begin{align*}
    V(F,0)&=\SetOf{\phi(e_i)-e_i}{i\in[d]}&&\text{and}& V(F,-1)&\subseteq\SetOf{-e_i}{i\in[d]}\,.
  \end{align*}
  So $\phi(e_i)\ne\0$ for all $i$, hence, by \autoref{lem:phi}, if $-e_i$ is a vertex of $P$, then
  $\phi(\phi(e_i)) = e_i$.  In this case, both $\phi(e_i)-e_i$ and $e_i-\phi(e_i)$ are vertices.

  As $\eta_{-2}=1$ there is a unique vertex $z\in V(F,-2)$.  \autoref{lemma:2} implies that $z_i \ge
  -1$ for all $i\in[d]$.  If $z_j>0$ for some $j\ge 2$, then $z\in V(\neigh(F,e_j),k)$ for some $k
  \le -3$, as $\neigh(F,e_j)$ has normal $\1-e_j$. Now $e_1\in \neigh(F,e_j)$, so $\neigh(F,e_j)$ is
  special. However, by Table~\ref{tab:eta} all possible $\eta$-vectors have $\eta_k=0$ for $k\le
  -3$, so $z_j\in\{-1,0\}$ for $j\ge2$.  Now consider
  \[
    v_P\ =\ e_1 \ =\ z + \sum_{i=1}^{d} e_i + \sum_{i=1}^{d} (\phi(e_i)-e_i) + \sum_{i \in I} -e_i
  \]
  for some index set $I\subset [d]$ with $|I|=d-3$.  Solving for $z$ yields
  \begin{equation}\label{eq:z_equal_sumOf}
    z\ =\ e_1 - \sum_{i=1}^d \phi(e_i) + \sum_{i\in I} e_i
  \end{equation}
  If $1\not\in I$ then $z_1\le 1$. If $1\in I$, then $-e_1$ is a vertex of $P$, so we know
  $\phi(\phi(e_1))=e_1$ by \autoref{lem:phi} and $e_1\in\image\phi$. So $e_1$ appears in both sums,
  and again $z_1\le 1$. Since $z_i \in \{ 0,-1 \}$ for all $i\ge2$ we know that $z$ has (up to
  lattice equivalence) one of the following forms: either $z = -e_i - e_j$ for distinct $i,j \in
  [d]$ or $z = e_1 - e_2 -e_3 -e_4$.\smallskip

  \noindent\framebox{Let $z = -e_i - e_j$.}
  Assume that $[d]\setminus I=\{a,b,c\}$. If $\phi(e_r)=e_a$ holds for some $r$ then
  Equation~\eqref{eq:z_equal_sumOf} implies $a \in \{ i,j \}$, and similarly for $b$ and $c$.  Since
  $a$, $b$, and $c$ are pairwise distinct we have $\{a,b,c\}\not\subseteq\{i,j\}$. So at least one
  of the indices cannot appear in the image of $\phi$ and we may assume that $e_a\not\in\image\phi$.
  So $P$ can be written as a skew bipyramid with apices $e_a$ and $\phi(e_a)-e_a$ over some
  $(d{-}1)$-dimensional polytope $Q$ with $3d-4$ vertices.  The polytope $Q$ is contained in linear
  hyperplane $x_a=0$.  \autoref{lem:Fano=Bipyramid_over_Fano} implies the claim.\smallskip

  \noindent\framebox{Let $z = e_1 - e_2 -e_3 -e_4$.}
  In this case $-e_2,-e_3,-e_4 \notin P$ by Equation~\eqref{eq:z_equal_sumOf}.  Let
  $J=\{2,3,4\}$. Then $I=[d]\setminus J$ and $J\subseteq \image \phi$.  We further distinguish two
  cases.
  \begin{itemize}
  \item Assume first that there are indices $k,\ell\in J$ with $\phi(e_k)=e_\ell$.  We may assume
    $k=3$ and $\ell=4$. By Equation~\eqref{eq:z_equal_sumOf} $\phi$ is bijective, so $\phi(e_2) \ne
    e_4$. If also $\phi(e_2)\ne e_3$ then the primitive vector
      \[
        u\ =\ \bigl(2,0,-1, 1, 2, \ldots, 2\bigr)\,.
      \]
      defines a valid inequality $\langle u,x\rangle\le 2$ for all $x\in P$.  The $d-1$ vertices
      \[
        z\,,\  e_1\,,\  \phi(e_2)-e_2 \,,\  -e_3 + e_4\,,\  e_5\,,\  \ldots\,,\  e_d
      \]
      satisfy $\langle u,x\rangle=2$, so $P$ would not be reflexive.  If on the other hand
      $\phi(e_2)=e_3$ then by \autoref{lem:2times_neighboring} the set
      \[
      G \ := \ \conv\left( \SetOf{e_k}{k\ne 2,3} \;\cup\; \{ (-e_2+e_{3}),\, (-e_3+e_{4}) \} \right)
      \]
      is a facet of $P$ with normal $u_G = \1 -2e_2 -e_3$.  But $\langle u_G, z \rangle = 1$, a
      contradiction.

    \item It remains to consider the case that $\phi(e_k)\ne e_\ell$ for all $k,\ell\in J$.  The
      vector
      \[
      u' \ =\ (1,0,0,0,1,\dots,1)\,,
      \]
      defines a valid inequality $\langle u',x\rangle\le 1$. But then the $d+1$ vertices
      \[
      e_1\,,\  \phi(e_2)-e_2 \,,\  \phi(e_3)-e_3 \,,\  \phi(e_4)-e_4 \,,\  e_4\,,\  \ldots\,,\  e_d\,,
      \]
      span a facet which is not a simplex.\qedhere
  \end{itemize}
\end{proof}

The remainder of this section deals with the situation where the $\eta$-vector reads $(d,d-1,d-1)$.
By the previous proposition it suffices to consider polytopes whose $\eta$-vectors of all special
facets agree.

\begin{lemma} \label{lem:dd-1d-1_opposite_of_vertexsum} Let $P$ be a $d$-dimensional \str polytope
  with exactly $3d-2$ vertices such that $v_P$ is a vertex and $\eta=(d,d-1,d-1)$ for all special
  facets of $P$.
  Then there is a special facet $F$ with $\opp(F,v)\! \ne \opp(F,v_P)$ for all $v\in \Vert F\setminus
  \{ v_P \}$.
\end{lemma}

\begin{proof}
  Pick any special facet $G$ and let $\phi:=\phi^G$.  If there is a vertex $w$ in $G$ with
  $\opp(G,w) \in V(G,-1)$ then we may take $F=G$, as no two vertices of $G$ share the same opposite
  vertex.

  So assume that $\opp(G,w)\in V(G,0)$ for all $w\in\Vert G$.  Let $w\in\Vert{G}\setminus\{v_P\}$.
  Then the neighboring facet $H:=\neigh(G,w)$ is special as it contains $v_P$.  Suppose that the
  vertex $w$ is on level $-1$ with respect to $H$.  Then, as above, we may take $F=H$.

  We will refute all remaining cases.  In view of our argument so far we can assume that for all
  vertices $w\in\Vert{G}\setminus\{v_P\}$ we have $\langle u_{\neigh(G,w)},w\rangle=0$; other
  choices for this value are ruled out by our assumption that the $\eta$-vectors of all special
  facets are equal to $(d,d-1,d-1)$.  In this case we can assume by \autoref{prop:classificationOf0and-1}
  that $\Vert G$ is a lattice basis, and, up to unimodular transformation, $G=\conv\{e_1,e_2,\ldots,
  e_d \}$ as well as $v_P = e_1$ and, for some indices $a,b\not\in\{1,2\}$,
  \[
  x \ := \ \opp(G,e_1) \ = \ \opp(G,e_2) \ = -e_1 -e_2 +e_a +e_b\,.
  \]
  If $a=b$, then $x = -e_1 -e_2 +2 e_a$. So $H':=\neigh(G,e_a)$ is a special facet with standard
  facet normal $u_{H'} = \1 -e_a$.  However, this means that $\langle u_{H'} , x\rangle = -2$ in
  contradiction to $\eta^{H'}=(d,d-1,d-1)$. So $a\ne b$.
  In view of \autoref{prop:classificationOf0and-1} we know
  \[
  V(G,-1) \ \subseteq \ \{ -e_1,\, -e_2,\, \ldots,\, -e_d \} \;\cup\; \SetOf{-e_1-e_2+e_r}{r\in[d]}\, .
  \]
  Assume that the set $V(G,-1)$ contains a vertex $z=-e_1-e_2+e_r$ for some $r\ne 1,2$.  Then,
  similar to the case above, the neighboring facet $\neigh(G,e_r)$ has standard facet normal
  $\1-e_r$ implying that $z$ is at level minus two with respect to that facet.  Again this is
  impossible.  We conclude that $V(G,-1) \subseteq \SetOf{-e_i}{i\in[d]}$.  Now $v_P=e_1$ requires
  that $-e_a$ and $-e_b$ both are vertices of $P$ and that $e_a,e_b \notin \image \phi$.  By
  \autoref{lem:2times_neighboring}
  \[
  H'' \ := \ \conv\left( \SetOf{e_k}{k\ne a,b} \;\cup\; \{ \phi(e_a)-e_a,\, \phi(e_b)-e_b \} \right)
  \]
  is a facet which is also special, since $e_1 \in H$.  But $x \in V(H'',-2)$,  a
  contradiction.

  Summarizing, this shows that we cannot have $\langle u_{\neigh(G,w)},w\rangle=0$ for all
  $w\in\Vert{G}\setminus\{v_P\}$.  So there is no vertex $v$ of $G$ with $\opp(G,v_P) = \opp(G,v)
  \in V(G,0)$.  We may thus take $F=G$.
\end{proof}

We also need the following variation of the previous lemma.

\begin{lemma}\label{lem:unique_opp_at_level-1}
  Let $P$ be a $d$-dimensional \str polytope with $3d-2$ vertices
  such that $v_P$ is a vertex and $\eta=(d,d-1,d-1)$ for all special facets of $P$.
  Then for some special facet $F$ there is  $v\in \Vert F$ with $\opp(F,v)\in V(F,-1)$.
\end{lemma}

\begin{proof}
  We prove this by contradiction, so assume that $\opp(G,v)\in V(G,0)$ for all special facets $G$
  and vertices $v\in \Vert G$.  Fix a special facet $G$ according to
  \autoref{lem:dd-1d-1_opposite_of_vertexsum}, so $\opp(G,v_P)\ne \opp(G,v)$ for all vertices $v$ of
  $\Vert G\setminus\{v_p\}$. By
  \autoref{prop:classificationOf0and-1}.\eqref{prop:classificationOf0and-1:3} we can find a lattice
  transformation such that $G=\conv\{ e_1,e_2,\ldots, e_d \}$ and with $\phi:=\phi^G$ we have
  \begin{align}
    V(G,0) \ &= \ \SetOf{ \phi(e_i)-e_i }{ i\in [d] \setminus\{ 1,2 \} } \;\cup\; \{ -e_1 -e_2 +e_a+e_b\}\notag\\
    V(G,-1) \ &\subseteq \ \{ -e_1,\, -e_2,\,\ldots, -e_d \} \;\cup\; \SetOf{-e_1-e_2+e_r}{r\in[d]}
    \label{eq:a}
  \end{align}
  with $a,b \not\in \{ 1,2 \}$. By our choice of $G$ we know that $v_P$ is distinct from $e_1$ and $e_2$.

  Assume first that $a\ne b$. Up to relabeling we can assume $v_P \ne e_a$ and $\phi(e_a)\ne e_1$.
  So $H:= \neigh(G,e_a)$ is still special. Further, $u_{H,e_1} = e_1$ and $u_H = \1 - e_a$, so we
  observe that $e_a\in V(H,0)$. Now the assumption $\eta^H = (d,d-1,d-1)$ implies $|V(G,0)\cap
  V(H,-1)|=|V(G,-1)\cap V(H,0)|$.  Hence, $V(G,-1)\cap V(H,0)\subseteq\{-e_a\}$ and $-e_1 -e_2 +e_a
  +e_b \in V(G,0)\cap V(H,-1)$ shows $e_a \notin \image \phi$. So
  \[
  V(H,0) \ = \ \SetOf{ \phi(e_i)-e_i }{ i\in [d] \setminus\{ 1,2, a \} } \cup \{e_a, -e_a\}\,.
  \]
  Therefore, $\langle u_{H,e_1}, y \rangle \ge 0$ for any $y\in V(H,0)$.  Hence, $\opp(H,e_1)\notin
  V(H,0)$ by \autoref{prop:nill:5.5}(\ref{prop:nill:5.5:2}) and we can choose $F=H$.

  We are left with $a = b$.  This implies $v_p = e_a = e_b$ as otherwise $\eta^H_{-2}\ne 0$ for the
  special facet $H:=\neigh(G,e_a)$.  We may assume that $a=3$. If $x:=-e_1-e_2+e_r\in V(G,-1)$ for
  some $r\ge 4$, then, as $H:=\neigh(G,e_4)$ is special, evaluating $\langle u_H,x\rangle =\langle
  \1-e_r, x\rangle=-2$ contradicts the assumption $\eta^H=(d,d-1,d-1)$.  Further, there neither
  exists a vertex $x\in V(G,-1)$ with $\langle u_{G,v_p}, x \rangle >0$ since otherwise
  \[
  \langle u_{G,v_P},v_P\rangle \ = \ \sum_{v\in \Vert P} \langle u_{G,v_P}, v \rangle \ \ge \ 2 \ >
  \ 1 \, .
  \]
  By Equation~\eqref{eq:a} we have $V(G,-1) = \smallSetOf{-e_i}{i\in[d] \setminus \{ k \}}$ for
  some $k\in[d]$. We may assume that $k\ne 1$.  So
  \begin{align*}
    u \ := \ (2,-2,1,2,\dots,2)
  \end{align*}
  induces an inequality $\langle u,y\rangle \le 2$ valid for all $y\in P$.  Furthermore, the
  vertices $e_1,e_4,e_5,\ldots, e_d, -e_2$, and $-e_1 -e_2 +2e_3$ satisfy this with equality, so $u$
  defines a facet, contradicting the reflexivity of $P$. This proves the claim.
\end{proof}
The previous result can now be extended to a characterization.
\begin{proposition}\label{prop:specialisvertex}
  Let $P$ be a $d$-dimensional \str polytope with exactly $3d-2$
  vertices, where $d\ge 4$, such that $v_P$ is a vertex and $\eta^G=(d,d-1,d-1)$ for every special
  facet $G$ of $P$.

  Then $P$ is lattice equivalent to a (possibly skew) bipyramid over a $(d{-}1)$ dimensional smooth
  Fano polytope with $3d-4$ vertices.
\end{proposition}

Once again we have $3d-4=3(d-1)-1$, and the possible bases of the bipyramids are classified in
\autoref{thm:3d-1}.

\begin{proof}
  By \autoref{lem:unique_opp_at_level-1} and \autoref{prop:classificationOf0and-1} we may assume
  $F=\conv\{e_1,e_2,\dots,e_d\}$ is a special facet with (using $\phi:=\phi^F$)
  \begin{align*}
    V(F,0) \ &=\ \SetOf{\phi(e_i) - e_i}{i \in [d]\setminus \{ a \}}\\
    \begin{split}
      V(F,-1) \ &\subseteq\ S\:=\  \{ -e_1,\,-e_2,\,\ldots,\, -e_d
      \}\\ &\qquad\qquad\qquad\;\cup\;\SetOf{-2e_a-e_r+e_s+e_t}{r,s,t \in [d] \text{ p.d.},\, r\ne
        a} 
    \end{split}
  \end{align*}
  for some unique index $a\in [d]$. Up to relabeling we can assume $v_P = e_1$. Let
  $x=-2e_a-e_r+e_s+e_t$ for some pairwise distinct $r,s,t$ with $r\ne a$. If $s\ne 1,a$, then $x$
  would lie on level $-2$ for the special facet $\neigh(F,e_s)$. This is excluded by assumption. The
  same holds for $t$.  So without loss of generality we may assume that $s=1$ and $t=a$, that is, $x
  = -e_a -e_r +e_1$ for some $r \in [d]$ and $r\ne 1,a$. We distinguish between $a=1$ and $a\ne 1$.
  \smallskip

  \noindent\framebox{Let $a=1$.} Then $V(F,-1) \subseteq S':= \{ -e_1,-e_2, \ldots, -e_d
  \}$ and $z := \opp(F,e_1) = \opp(F,e_a) \in V(F,-1)$. This implies that $\langle u_{F,e_1}, z
  \rangle < 0$. The only vertex in $S'$ that satisfies this is $-e_1$, so $z=-e_1$. Since
  $\eta_{-1}=d-1$ up to relabeling we have
  \[
  V(F,-1) \ = \ \{ -e_1,\, -e_2,\, \ldots,\, -e_{d-1} \}\,.
  \]
  $v_P=\0$ requires that $e_d \notin \image \phi$, and thus $P$ is a skew bipyramid with apices
  $e_d$ and $\phi(e_{d}) - e_d$ over some $(d-1)$ dimensional polytope $Q$ with $3d-4$ vertices. By
  \autoref{lem:Fano=Bipyramid_over_Fano} we know that $Q$ is again \str.  This is the claim.

  \smallskip

  \noindent\framebox{Let $a\ne 1$.} Let $z:=\opp(F,e_a) \in V(F,-1)$. As before we have
  $\langle u_{F,e_a}, z \rangle <0$.  Among the points in $S$ this is only satisfied by $-e_a$ and
  $-e_a-e_r+e_1$ for some $r\in[d]\setminus \{ 1,a \}$. In either case the facet normal of $F' =
  \neigh(F,e_a)$ is $u_{F'} = \1 -2e_a$, so only one of those points can be in $V(F,-1)$. Hence,
  $V(F,-1) \subseteq \{ -e_1,-e_2,\ldots, -e_d \} \cup \{ z \}$. Using $v_P=\0$ we see that $e_a
  \not\in \image\phi$.

  If $z = -e_a$ we conclude that $P$ is a proper bipyramid over a $(d-1)$-dimensional smooth Fano
  polytope $Q$ with $3d-4$ vertices. The polytope $Q$ is the intersection of $P$ with the hyperplane
  $x_a = 0$.

  If $z = -e_a -e_r +e_1$ for some $r\in[d]\setminus\{ 1,a \}$ and $\phi(e_r)=e_1$ we get a skew
  bipyramid with apices $e_a$ and $z$ (because the line segment between $e_a$ and $z$ intersects
  with the hyperplane $x_a = 0$ in the interior of $P$).

  We will show that the remaining case $z = -e_a -e_r +e_1$ for some $r\in[d]\setminus\{ 1,a \}$ and
  $\phi(e_r)\ne e_1$ does not occur.  We claim that $-e_r \notin P$. Let $F^{(r)} =
  \neigh(F,e_r)$. The facet normal reads $u_{F^{(r)}} = \1 - e_r$, so that $z,\, -e_r \in H(F^{(r)},
  0)$. Further, we have $u_{F^{(r)},\phi(e_r)} = \phi(e_r)+e_r$. \autoref{lemma:2} together with
  $\langle u_{F^{(r)},\phi(e_r)}, z \rangle = -1$ show that $z=\opp(F^{(r)},\phi(e_r))$. If $-e_r\in
  P$, then, by the same argument, $\opp(F^{(r)},\phi(e_r))=-e_r$. Hence, $-e_r\not\in P$, and
  \[
  V(F,-1) \ = \ \SetOf{-e_i}{i\in[d] \setminus\{a,r\}} \cup \{ z \}\, .
  \]
  Now $e_a\not\in \image\phi$, $e_r \in \image\phi$ and \autoref{lem:phi} shows that
  $\phi(\phi(e_i)) = e_i$ for all $i\in[d] \setminus \{ a \}$.

  Let $u:=\1 -e_1 -3e_r -2\phi(e_{r})$. Inspecting $V(F,0)$ and $V(F,-1)$, we see that the
  inequality $\langle u, x \rangle \le 1$ is valid for $P$ (recall that $\eta^F_{k}=0$ for $k\le
  -2$).  $u$ induces a face
  \[
  G \ := \ \conv\Big( (\Vert{F}\setminus\{ e_1,\, e_r,\, \phi(e_r) \}) \;\cup\; \{ z,\,
  -\phi(e_r),\, \phi(e_r)-e_r,\, \phi(e_1)-e_1 \} \Big)\,
  \]
  that is actually a facet. However, it contains $d+1$ vertices, so it is not a simplex.  This is
  the desired contradiction.
\end{proof}

\subsection{Polytopes with Eccentricity \texorpdfstring{$0$}{0}}

This is equivalent to $v_P=\0$, which makes every facet a special facet.  This situation is the most
difficult by far.  We start out with a general characterization of the centrally symmetric case.
For this result we neither make any assumption on the number of vertices nor on the terminality.

\begin{proposition}\label{prop:centrallySymmetric}
  Let $P$ be a simplicial and reflexive polytope. Then $P$ is centrally symmetric if and only if $v_P=\0$ and
  $\eta^G_\ell =  0$ for every facet $G$ and any $\ell\le -2$. In other words: the polytope $P$ is centrally
  symmetric if and only if it's lattice width in each facet direction is equal to $2$.
\end{proposition}
\begin{proof}
  Assume first that $v_P=\0$ and $\eta^G_\ell = 0$ for every facet $G$ and any $\ell\le -2$.  Let
  $\langle u, x \rangle \le 1$ be a facet defining inequality.  Our assumption on the $\eta$-vectors
  implies that $\langle -u, x \rangle \le 1$ is a valid inequality. Take a vertex $v$ and look at
  its antipode $-v$.  If $-v\not\in P$ then there is a facet defining inequality which separates
  $-v$ from $P$.  By the above argument there would be a valid inequality separating $v$ from $P$,
  which is impossible.  Therefore, the point $-v$ is contained in $P$. Further, since $-v$ satisfies
  at least $d$ linearly independent valid inequalities with equality it must be vertex.  The
  converse direction is obvious.
\end{proof}

The centrally symmetric smooth Fano polytopes are listed in Theorem~\ref{thm:pseudo_symmetric} which
sums up results of Voskresensk\u\i{} and Klyachko~\cite{Voskresenskij1985}, Ewald~\cite{Ewald88,Ewald96}
and Nill~\cite{1067.14052,0511294}.

\begin{corollary}\label{cor:centrally_symmetric}
  Let $P$ be a $d$-dimensional polytope \str polytope with exactly $3d-2$ vertices. If $P$ is 
  centrally symmetric then it is lattice equivalent to
  \begin{enumerate}
    \item\label{cor:centrally_symmetric:1} a double proper bipyramid over $P_6^{\oplus \frac{d-2}{2}}$ or
    \item $\HSBC{4} \oplus P_6^{\oplus \frac{d}{2} - 2}$.
  \end{enumerate}
\end{corollary}

\begin{proof}
  A centrally symmetric, simplicial, and reflexive polytope $P$ is a direct sum of centrally
  symmetric cross polytopes and del Pezzo polytopes. A $k$-dimensional cross polytope has $2k$
  vertices. A $k$-dimensional del Pezzo polytope has $2k+2$ vertices (and $k$ is even by definition
  of del Pezzo polytopes), while a $k$-dimensional pseudo del Pezzo polytope has $2k+1$ vertices
  (and again, $k$ must be even). The latter is not centrally symmetric, so no direct sum involving
  it will be.

  We need to find those direct sums of these three types of polytopes that have $3d-2$ vertices in
  dimension $d$. On average, for each dimension the polytope must have $3-\frac{2}{d}$
  vertices. This is only possible if at most $2$ of the summands are not $2$-dimensional del Pezzo
  polytopes. This leaves us with direct sums of $2$-dimensional del Pezzo polytopes with one
  $2$-dimensional cross polytope, or $\HSBC{4}$.
\end{proof}

We first establish a further restriction on the $\eta$-vectors.  This says that the $\eta$-vector in
the fourth column of Table~\ref{tab:eta} does not occur.  Notice that, if $d=2$ in
\autoref{cor:centrally_symmetric} then $d-2=0$ and $P^{\oplus \frac{d-2}{2}}$ is the origin.  In
this case the only polytope of type \eqref{cor:centrally_symmetric:1} is the regular cross-polytope
$P_{4a}$ shown in Figure~\ref{fig:P4a}; this is centrally symmetric.

\begin{proposition}\label{prop:novertexat-3}
  Let $P$ be a simplicial and reflexive $d$-polytope with exactly $3d-2$ vertices satisfying
  $v_P=\0$.  Then $\eta^G_\ell=0$  for all $\ell\le -3$ and for each facet $G$.
\end{proposition}

\begin{proof}
  We fix a facet $G$ and abbreviate $\eta=\eta^G$ and $\phi:=\phi^G$.  Now suppose that
  $\eta_{-\ell}>0$ for some $\ell\le -3$.  According to \autoref{tab:eta} then $\ell=-3$ and
  $\eta=(d,d,d-3,0,1)$.  Let $z$ be the unique vertex in $V(G,-3)$.  We are aiming at a
  contradiction.

  By \autoref{prop:classificationOf0and-1}\eqref{prop:classificationOf0and-1:1} we may assume that
  $G=\conv\{ e_1,e_2,\ldots, e_d \}$ and (up to relabeling)
  \begin{align*}
    V(G,0)\ &=\ \SetOf{\phi(e_i)-e_i}{i\in[d]}&&\text{and}& V(G,-1)\ &=\ \{-e_1, -e_2, \dots,
    -e_{d-3}\}\,.
  \end{align*}

  In this situation \autoref{lem:phi} shows $\phi(\phi(e_i))=e_i$ for $1\le i\le d-3$ and hence
  $\smallSetOf{e_i}{i\in[d-3]} \subseteq \image \phi$.  The condition $v_P=\0$ tells us that
  \begin{equation}
    z\ + \ \sum_{i=1}^d \phi(e_i) \ +\ \sum_{i=1}^{d-3} -e_i \ =\ 0\,.\label{eq:prop:prop:novertexat-3:1}
  \end{equation}
  Hence, $z = -e_i-e_j-e_k$ for some $i,j,k\in [d]$.  Suppose that $j=k$.  Then the midpoint
  \[
  \frac{1}{2}z + \frac{1}{2}e_i \ = \ \frac{1}{2}(-e_i-2e_j)+\frac{1}{2}e_i \ = \ -e_j
  \]
  of the line segment between $z$ and $e_i$ is a non-zero lattice point in $P$ that is not a
  vertex. This contradicts terminality of $P$.  We conclude that the indices $i,j,k$ are pairwise
  distinct.  Two cases may occur.\smallskip

  \noindent\framebox{Let $z\ne-e_{d-2}-e_{d-1}-e_d$.}
  Choose $a\in\{d-2,d-1,d\}\setminus\{i,j,k\}$. Then $-e_a\not\in P$.  Hence, by
  Equation~\eqref{eq:prop:prop:novertexat-3:1} we know $\image \phi = \{ e_1,e_2,\ldots, e_{d-3} \}
  \cup \{ e_i,e_j,e_k \}$ and $e_a \notin \image \phi$. Thus, $x_a=0$ for every vertex $x\in V(G,0)
  \setminus \{ \phi(e_a)-e_a \}$.

  We conclude that $P$ is a skew bipyramid with apices $e_a$ and $\phi(e_a)-e_a$ over $Q=
  \smallSetOf{x\in P}{x_a=0}$, and $Q$ is a \str $(d{-}1)$-polytope
  with $3d-4$ vertices.  The face $H:=G\cap Q$ is a facet of $Q$.  With respect to $H$ the vertex
  $z\in Q$ is still on level $-3$. However, we can check using \Obro's classification given in
  \eqref{eq:obro:even}, \eqref{eq:obro:odd-not-skew}, and \eqref{eq:obro:odd-skew} that no facet
  (also the non-special ones) of a simplicial, terminal, and reflexive polytope with $3d-1$ vertices
  has a vertex on level $-3$. To check this, observe that it suffices to look at each summand of the
  given representation separately. Those are $P_5$, $P_6$, and the proper and skew bipyramid over
  $P_6$. Thus, there is no such polytope $Q$ that could serve as a basis for the skew bipyramid $P$.

  \smallskip

  \noindent\framebox{Let $z=-e_{d-2}-e_{d-1}-e_d$.}
  In this case Equation~\eqref{eq:prop:prop:novertexat-3:1} implies that $\phi$ is bijective.  We
  have $\phi(\phi(e_i))=e_i$ for $1\le i\le d-3$ as well as $\phi(e_i)\not\in\{ 0, e_i\}$ for all $i
  \in [d]$.  Up to relabeling we are left with three possibilities:
  \begin{itemize}
  \item $\phi(e_{d-2}),\phi(e_{d-1}),\phi(e_d) \in \{ e_1,e_2,\ldots, e_{d-3} \}$: Then $P$ is the
    direct sum of two polytopes $Q\subset \RR^{6}$ and $R \subset \RR^{d-6}$, where $Q$ is the
    convex hull of the 16 points
    \begin{gather*}
      e_1\,,\ e_2\,,\ \ldots\,,\ e_6\\
      \pm(e_1 - e_4)\,,\ \pm(e_2 - e_5)\,,\ \pm(e_3 - e_6)\\
      -e_1\,,\ -e_2\,,\ -e_3\\
      -e_4-e_5-e_6 \, .
    \end{gather*}
    However, $Q$ is not simplicial as the vector $(1,1,-1,0,1,-2)$ induces a facet with seven
    vertices.
  \item $\phi(e_{d-2}) \in \{ e_1,e_2,\ldots, e_{d-3} \}$, $\phi(e_d)=e_{d-1}$, and
    $\phi(e_{d-1})=e_d$: Then $P$ is the direct sum of two polytopes $Q\subset \RR^{4}$ and $R
    \subset \RR^{d-4}$, where $Q$ is the convex hull of the ten points
    \[
       e_1\,,\ e_2\,,\ e_3\,,\ e_4\,;\quad
       \pm(e_1 - e_2)\,,\ \pm(e_3 - e_4)\,;\quad
       -e_1\,;\quad -e_2-e_3-e_4 \,.
    \]
    Again $Q$ is not simplicial as the vector $(-1,-2,0,1)$ induces a facet with four vertices.
  \item $\phi(e_{d-2})=e_{d-1}$, $\phi(e_{d-1})=e_d$, and $\phi(e_d)=e_{d-2}$: Then $P$ is the
    direct sum of two polytopes $Q\subset \RR^{3}$ and $R \subset \RR^{d-3}$, where $Q$ is the
    convex hull of the ten points
    \[
       e_1\,,\ e_2\,,\ e_3\,;\quad
       \pm(e_1 - e_2)\,,\ \pm(e_2 - e_3)\,,\ \pm(e_3 - e_1)\,;\quad
       -e_1-e_2-e_3 \, .
    \]
    For $u=(-1, 2, -4)$ the inequality $\langle u,x\rangle \le 3$ induces a triangular facet of $Q$
    with vertices $-e_1+e_2$, $-e_3+e_1$, and $-e_1-e_2-e_3$.  This means that $Q$ is not
    reflexive. \qedhere
  \end{itemize}
\end{proof}

In view of Table~\ref{tab:eta} the preceding result leaves $(d,d,d-4,2)$, $(d,d-1,d-2,1)$ or
$(d,d-2,d)$ as choices for the $\eta$-vector of any special facet. As \autoref{prop:centrallySymmetric} and \autoref{cor:centrally_symmetric} deal with the case that every facet has $\eta$-vector $\eta=(d,d-2,d)$ only two cases remain.

We will look at the situation where $v_P=\0$ and there is a special facet $F$ with $\eta=(d,d-1,d-2,1)$. Since
$\eta_{0}=d-1$ we know that there is at least one vertex $v\in F$ with $\phi(v)=\0$. Otherwise the opposite vertex
of $v$ would be $\phi(v) -v \in V(F,0)$ for every choice of $v$ which would imply $\eta_{0}=d$. One technical
difficulty in our proof is that we will have to distinguish whether the vertex $\opp(F,v)$ lies on level $0$, $-1$
or $-2$ with respect to $F$. If $\opp(F,v)$ lies on a level below $0$, all the other vertices $w$ of $F$ are good
with $\phi(w)\ne\0$.  These cases are easier, and we will deal with them first. The remaining case where $\eta =
(d,d,d-4,2)$ will be reduced to one of the cases above.

\begin{lemma}\label{lem:dd-1d-21a}
  Let $P$ be a \str $d$-polytope with exactly $3d-2$ vertices such that $v_P=\0$ and there is a
  special facet $F$ with $\eta=(d,d-1,d-2,1)$.  If the unique vertex at level $-2$ with respect to
  $F$ is opposite to some vertex in $F$ then $P$ is lattice equivalent to a skew bipyramid over
  a $(d-1)$-dimensional smooth Fano polytope with $3d-4$ vertices.
\end{lemma}
\begin{proof}
  By \autoref{prop:classificationOf0and-1}\eqref{prop:classificationOf0and-1:3} we may assume $F =
  \conv\{ e_1, e_2, \ldots, e_d \}$, and, up to relabeling, $z := \opp(F,e_1) \in V(F,-2)$ as well
  as (with $\phi:=\phi^F$)
  \begin{align}
    V(F,0)\ &=\ \SetOf{\phi(e_i)-e_i}{i\in[d]\setminus\{ 1 \}}\notag\\
    V(F,-1)\ &=\ \SetOf{-e_i}{i\in[d]} \;\cup\; \SetOf{-2e_1-e_r+e_s+e_t}{r,s,t\in[d] \text{ p.d.},\, r\ne 1}\label{eq:dd-1d-21a:1}
  \end{align}

  If $z_i=\langle u_{F,e_i}, z \rangle$ is positive for some $i\ge 2$, then $z\in
  V(\neigh(F,e_i),-3)$. But $\neigh(F,e_i)$ is special, so this contradicts
  \autoref{prop:novertexat-3}.  Moreover, $z=\opp(F,e_1)$ implies $z_1<0$ by
  \autoref{lemma:2}. Hence $z=-e_1 -e_k$ for some $k\in[d]$, and $k\ne 1$ as $z$ is primitive.  Up
  to relabeling we may assume that $k=2$. The standard facet normal of the facet $F':=\neigh(F,e_1)$
  is $u_{F'}=(-2,1,\ldots, 1)$. Evaluating this on the right hand side of
  Equation~\eqref{eq:dd-1d-21a:1} shows that $V(F,-1)$ is already contained in the reduced set
  \[
  V(F,-1)\ \subseteq\ \{ -e_2,\, -e_3,\, \ldots,\, -e_d \}\, .
  \]
  Using the fact that $v_P=0$ we know that, in particular, the first coordinates of all vertices sum
  to zero.  We conclude that the map $\phi$ does not attain the value $e_1$. Thus, $P$ is a skew
  bipyramid with apices $e_1$ and $z=-e_1-e_2$ over the polytope $Q=\smallSetOf{y\in P}{y_1=0}$.
  Those are classified in \autoref{thm:3d-1}.
\end{proof}

\begin{example}\label{example:B-again}
  The facet $H=\conv\{e_2, e_3, e_4, e_2-e_1\}$ of the $4$-polytope $B$ from \autoref{example:eta}
  has facet normal $u_H=-e_2-e_3-e_4=\1-e_1$. Hence, it satisfies the conditions of
  \autoref{lem:dd-1d-21a}: the unique vertex $e_1-e_2-e_3$ at level $-2$ with respect to $H$ is
  opposite to $e_3$.  As pointed out previously, $B$ is a skew bipyramid over~$P_5$.
\end{example}

The following gives a, slightly technical, sufficient condition for the existence of a pair of
distant vertices, unless the polytope is a skew bipyramid.

\begin{lemma}\label{lem:dd-1d-21b}
  Let $P$ be a \str $d$-polytope with exactly $3d-2$ vertices such that $v_P=\0$ and $F = \conv\{
  e_1, e_2, \ldots, e_d \}$ is a special facet with $\eta = (d, d-1, d-2,1)$. If there exist
  pairwise distinct indices $r,s,t\in[d]$ with $\opp(F,e_r) = -e_r-e_s+e_t$ then the vertices $e_r$
  and $\opp(F,e_r)$ are distant, or $P$ is a skew bipyramid over some smooth Fano $(d{-}1)$-polytope
  with $3d-4$ vertices.
\end{lemma}
\begin{proof}
  We fix $\phi:=\phi^F$. Up to relabeling we may assume that $(r,s,t)=(1,2,3)$. Then $\phi(e_1)=\0$
  as $x:=\opp(F,e_1)=-e_1-e_2+e_3$ does not lie on level $0$. According to
  \autoref{prop:classificationOf0and-1} we get:
  \begin{equation}\label{eq:dd-1d-21b:V0-1}
    \begin{split}
      V(F,0)  \ &=\ \{ \phi(e_2)-e_2,\, \phi(e_3)-e_3,\, \dots,\, \phi(e_d)-e_d \}\\
      V(F,-1) \ &\subseteq \ \SetOf{-e_i}{i\in[d]}\cup \{ x \}
    \end{split}
  \end{equation}
  Let $z$ be the unique vertex at level $-2$.  \autoref{prop:novertexat-3} implies that $z$ only
  has non-positive coefficients. Otherwise we would look at the special facet $\neigh(F,e_i)$ with
  $i$ being the index of the positive coefficient. \autoref{lemma:2} tells us that in this case $z$
  must lie on a level below $-2$ with respect to $\neigh(F,e_i)$. So we can assume $z= -e_k -
  e_\ell$ for some indices $k, \ell$. Terminality and $e_k \in \conv\{ z, e_\ell -e_k\}$ implies
  $\phi(e_k)\ne e_\ell$. Similarly, $\phi(e_\ell)\ne e_k$.

  We give a brief outline of the proof.  By \autoref{lemma:nill_sumDistant}, the vertices $e_1$ and
  $\opp(F,e_1)$ are distant if and only if $e_1+\opp(F,e_1)=-e_2+e_3$ is a vertex of~$P$ and
  $\phi(e_2)=e_3$. So our goal is to show that if $\phi(e_2) \ne e_3$ the polytope $P$ is a skew
  bipyramid. We start out by reducing the possible choices for $\phi(e_2)$ and $\phi(e_3)$. In
  particular, we will show that $\phi(e_2) \ne e_1$ and $\phi(e_3) \in \{ e_1, e_2 \}$. Then we
  will distinguish between the two choices for $\phi(e_3)$. By counting and using that $v_P=\0$ we
  will be able to deduce the precise shapes of $V(F,-1)$ and $z$.  This will give a complete
  description of the polytope~$P$. A direct computation will finally show that $P$ is not
  simplicial, unless it is a skew bipyramid.

  Suppose $\phi(e_2)=e_1$.  Then $e_1-e_2=\phi(e_2)-e_2$ is a vertex of $P$. If $e_1$ and $x$ are
  not distant, then there is a facet $G$ that contains $e_1$ and $x$.  Its facet normal $u_G$
  satisfies $\langle u_G,e_1\rangle=\langle u_G,x\rangle=1$ and $\langle u_G,e_1-e_2\rangle, \langle
  u_G, e_3\rangle\le 1$, so
  \begin{align*}
    1\ge \langle u_G, e_1-e_2\rangle\ =\ \langle u_g, e_1\rangle + \langle u_G, x\rangle+\langle
    u_G, e_1-e_3\rangle \ =\ 1+1+1-\langle u_G, e_3\rangle \ \ge\ 2\,,
  \end{align*}
  a contradiction.  Hence, $\phi(e_2)\ne e_1$.

  If $\phi(e_2)=e_3$ then $e_1+x=-e_2+e_3=\phi(e_2)-e_2$, in which case $e_1$ and $x=\opp(F,e_1)$
  are distant due to \autoref{lemma:nill_sumDistant}.  So we may assume that $\phi(e_2)\ne e_3$, and
  up to relabeling we can set $\phi(e_2) = e_4$.

  Next we will we show that $\phi(e_3) \in \{ e_1,e_2 \}$. Consider the facet
  $F^{(3)}=\neigh(F,e_3)$.  Its vertices are $e_i$ for $i\ne 3$ and $\phi(e_3)-e_3$, so the vertices
  of $F^{(3)}$ form a lattice basis.  Writing $x$ in this basis gives
  \[
  x  \ = \ -e_1 -e_2 - (\phi(e_3) - e_3) + \phi(e_3)\, .
  \]
  Since $x \in V(F^{(3)}, -2)$ we know from \autoref{prop:novertexat-3} that $x$ has no positive
  coefficient with respect to this basis
  (as above, we can see that otherwise $x$ is at level $-3$ for any neighboring facet corresponding to a coordinate with positive
  coefficient). Hence, $\phi(e_3)$ must cancel with one of $-e_1$ or $-e_2$, so $\phi(e_3) \in \{
  e_1, e_2 \}$.

  We determine the vertices in $V(F,-1)$.  The normal vector of $F^{(1)}=\neigh(F,e_1)=\conv(x,e_2,
  \ldots, e_d)$ is $u_{F^{(1)}} = \1 - 2e_1$. Hence, if $-e_1\in \Vert P$, then $-e_1\in F^{(1)}$
  and $F^{(1)}$ would contain more than $d$ vertices. So $-e_1\not\in \Vert P$.  Now consider
  $F^{(2)}=\conv(e_1, e_4-e_2, e_3, \ldots, e_d)$. Then $u_{F^{(2)}, e_4}=e_2+e_4$, so
  \begin{align*}
    \langle u_{F^{(2)}, e_4}, x\rangle \ &=\ -1& \langle u_{F^{(2)}}, x\rangle \ =\ 0\,.
  \end{align*}
  By~\autoref{prop:nill:5.5}.(\ref{prop:nill:5.5:2}) this implies $\opp(F^{(2)}, e_4)=x$, and
  \begin{align*}
    F^{(2,4)} \ :=\ \neigh(\neigh(F,e_2), e_4) \ = \ \conv\left( \SetOf{e_i}{i\in[d]\setminus\{ 2,4
        \}} \cup \{ (e_4-e_2),\, x\}\right)\,.
  \end{align*}
  Hence, $u_{F^{(2,4)}} = \1 -2e_2 -e_4$. As $\langle u_{F^{(2,4)}}, -e_2 \rangle=1$ and
  $-e_2\not\in \Vert{F^{(2,4)}}$ we know that $-e_2\not\in V(F,-1)$.  For the remaining part of our
  proof we will distinguish between $\phi(e_3)=e_1$ and $\phi(e_3)=e_2$.  \smallskip

  \framebox{Let $\phi(e_3)= e_1$.} We list the vertices already known to have a non-vanishing first
  coordinate:
  \[
  e_1 \,, \quad -e_1 -e_2 + e_3 \,, \quad e_1-e_3\, .
  \]
  From $v_P = \0$ we learn that there must exist at least one more vertex with a negative first
  coordinate. By \eqref{eq:dd-1d-21b:V0-1} and $-e_1\not\in P$ the only possibility is $k=1$, which
  means $z=-e_1 -e_\ell$.  This also implies that $\phi(e_i) \ne e_1$ for $i\ne 3$.  Now we list the
  vertices already known to have a non-zero third coordinate:
  \[
  e_3 \,, \quad e_1 -e_3 \,, \quad -e_1 -e_2 +e_3 \, .
  \]
  This tells us that there must exist another vertex with a negative third coordinate. There are two
  possibilities; either $-e_3\in \Vert P$ or $\ell=3$, i.e.\ $z =-e_1 -e_3$.  As
  $-e_3=\tfrac{1}{2}((-e_1-e_3)+(e_1-e_3))$ and $\phi(e_3)=e_1$ the polytope $P$ is not terminal in
  the latter case.  So $k\ne 3$,  $-e_3\in\Vert P$, and there is exactly one index $j\ne 1,2,3$ such
  that
  \[
  V(F,-1) \ = \ \SetOf{-e_i}{i\in[d]\setminus\{ 1,2,j \}} \,\cup\, \{ x \}\, .
  \]

  We list all the vertices known to have a non-vanishing $j$-th coordinate:
  \[
  e_j \,, \quad \phi(e_j) - e_j\, .
  \]
  If $e_j\not\in\image \phi$ these are the only two vertices with this property, and hence $P$ is a
  skew bipyramid with apices $e_j$ and $\phi(e_j)-e_j$.  So in the following we may assume that $e_j
  \in \image \phi$. Hence, $\ell=j$, that is, $z= -e_1 -e_j$.  We will distinguish between $j=4$ and
  $j\ne 4$.  The aim is to write $P$ as a sum of two polytopes $P= Q\oplus R$ where we can show that
  $Q$ is not simplicial.

  Consider first the case $j=4$. By \autoref{lem:phi} $\phi(\phi(e_i))\in \{\0, e_i\}$ if $-e_1\in
  \Vert P$, so $\phi(e_i)\ne e_2$ for $i\ne 1,2,4$, as $\phi(e_2)=e_4$. As $v_P=\0$ implies $e_2\in
  \image\phi$ we deduce $\phi(e_4)=e_2$. By the same lemma, and using $\phi(e_i)\ne e_1$ for $i\ne
  3$ as well as $-e_i\in \Vert P$ for $i\ge 5$ we obtain $\phi(\phi(e_i))=e_i$ for $i\ge 5$. Hence,
  $\phi(e_i)\not\in \{ e_1,e_2,e_3,e_4 \}$. We can write $P = Q \oplus P_6^{\oplus \frac{d-4}{2}}$
  where $Q$ is the convex hull of the ten points:
  \begin{gather*}
    e_1 \,,\ e_2 \,,\ e_3 \,,\ e_4 \,,\\
    e_4-e_2 \,,\ e_1-e_3 \,,\ e_2 - e_4 \,,\\
    -e_1-e_2+e_3 \,,\ -e_3 \,,\\
    -e_1 -e_4 \, .
  \end{gather*}
  $Q$ is not simplicial as $u = (1,-1,1,-2)$ induces a facet which is not a simplex. Hence, $P$
  would not be simplicial.

  So we may assume that $j \ne 4$. So, up to relabeling  $j=5$. We list the vertices known to have a
  non-zero fourth coordinate:
  \[
  e_4 \,, \quad \phi(e_4) - e_4 \,,\quad e_4 - e_2 \,, \quad -e_4 \, .
  \]
  From \autoref{lem:phi} it follows that $\phi(\phi(e_4))=e_4$. So the only possible way to maintain
  $v_P = \0$ is that $\phi(e_4)=e_2$. We have already seen that $\phi(e_j)\ne e_5$ for $j\le 4$. By
  assumption, $e_5 \in \image \phi$, so there is some $j\ge 6$ such that $\phi(e_j) = e_5$. Now
  $-e_j \in \Vert P$ implies $\phi(e_5) = e_j$, so $j$ is unique and we can assume that $j=6$.  This
  leaves us with $P= Q \oplus P_6^{\oplus\frac{d-6}{2}}$, where $Q$ is the convex hull of
  \begin{gather*}
    e_1 \,,\ e_2 \,,\ e_3 \,,\ e_4 \,,\ e_5 \,,\ e_6 \,,\\
    e_4-e_2 \,,\ e_1-e_3 \,,\ e_2-e_4 \,,\ e_6-e_5 \,,\ e_5 - e_6\,,\\
    -e_1-e_2+e_3 \,,\ -e_3 \,,\ -e_4 \,,\ -e_6 \,,\\
    -e_1 -e_5\,.
  \end{gather*}
  The polytope $Q$ is not simplicial as the vector $u = (1,-2,0,-1,-2,-1)$ induces a facet which is
  not a simplex.  This is a contradiction to $P$ being simplicial, and this concludes the case
  $\phi(e_3)=e_1$.

  \framebox{Let $\phi(e_3)= e_2$.} The argument is similar to the other case. Using
  \autoref{lem:phi} $-e_3 \in \Vert P$ would imply $\phi(e_2) = \phi(\phi(e_3))= e_3$ which
  contradicts the assumption $\phi(e_2) = e_4$.  Hence,
  \[
  V(F,-1) \ = \ \SetOf{-e_i}{i\in[d]\setminus\{ 1,2,3 \}} \,\cup\, \{ x \}\, .
  \]
  We list the vertices known to have a non-vanishing third coordinate:
  \[
  e_3 \,,\quad e_2 -e_3 \,,\quad -e_1 -e_2 +e_3\, .
  \]
  Again $v_P=\0$ implies that $k=3$, that is, $z = -e_3 -e_\ell$ for some $\ell$ and $e_3
  \not\in\image\phi$. We have seen above that $\phi(e_3) = e_2$ implies $\ell\ne 2$.

  If $\phi(e_4) = e_2$, then, using $\ell \ne 2$ and $-e_2\not\in \Vert P$, we have the following
  vertices with a non-zero second coordinate:
  \[
  e_2  \,,\quad e_4 -e_2  \,,\quad e_2 - e_4\,.
  \]
  This violates $v_P = \0$, so $\phi(e_4)\ne e_2$.

  We consider $\phi(e_4) = e_1$ and $\phi(e_4) \ne e_1$ separately.  In the former case, by counting
  all vertices with a non zero first coordinate
  \[
  e_1  \,,\quad -e_1 -e_2 +e_3  \,,\quad e_1 -e_4\, ,
  \]
  we see that $\ell=1$, that is, $z=-e_1 -e_3$. Therefore, $P=Q \oplus P_6^{\oplus \frac{d-4}{2}}$
  for the convex hull $Q$ of the points
  \begin{gather*}
    e_1\,,\ e_2\,,\ e_3\,,\ e_4 \,,\\
    e_4-e_2\,,\ e_2-e_3\,,\ e_1 - e_4\,,\\
    -e_1-e_2+e_3\,,\ -e_4\,,\\
    -e_1 -e_3\, .
  \end{gather*}
  The polytope $Q$ is not simplicial as $u = (1, -1, 1, 0)$ induces a facet which is not a simplex.
  Again, this contradicts that $P$ is simplicial.

  We are left with $\phi(e_4)\ne e_1$. As argued above $\phi(e_4)\ne e_2$ and as
  $\phi(\phi(e_4))=e_4$ due to \autoref{lem:phi} we have that $\phi(e_4)\ne e_3$. So we may assume
  $\phi(e_4)=e_5$. Examining all known vertices with a non-zero fourth coordinate, namely,
  \[
  e_4  \,,\quad e_5 - e_4  \,,\quad e_4 -e_5  \,,\quad e_4 - e_2  \,,\quad -e_4 \, ,
  \]
  we see that $z=-e_3-e_4$.  We can split $P$ into the sum $Q\oplus P_6^{\oplus \frac{d-5}{2} }$,
  where $Q$ is the convex hull of
  \begin{gather*}
    e_1\,,\ e_2\,,\ e_3\,,\ e_4\,,\ e_5\,,\\
    e_4-e_2\,,\ e_2-e_3\,,\ e_5 - e_4\,,\ e_4 -e_5\,,\\
    -e_1-e_2+e_3\,,\ -e_4\,,\ -e_5\,,\\
    -e_3 -e_4\,.
  \end{gather*}
  The polytope $Q$ is not simplicial as $u = (1,1,0,0,-1,0)$ induces a facet which is not a
  simplex. Once again, this is a contradiction to $P$ being simplicial. This concludes the case
  $\phi(e_3) \ne e_1$ and the entire proof.
\end{proof}

The following strengthens \autoref{lem:dd-1d-21a} by slightly relaxing the preconditions.

\begin{proposition}\label{prop:dd-1d-21}
  Let $P$ be a \str $d$-polytope with exactly $3d-2$ vertices such that $v_P=\0$. Suppose that $P$
  has a (special) facet $F$ with $\eta^F=(d,d-1,d-2,1)$ and there is a vertex $v$ of $F$ such that
  $\opp(F,v) \not\in V(F,0)$.  Then $P$ is a (possibly skew) bipyramid over some \str polytope of
  dimension $d-1$ with $3d-4$ vertices.
\end{proposition}

\begin{proof}
  If $\opp(F,v) \in V(F,-2)$ for some $v\in \Vert F$ we can apply \autoref{lem:dd-1d-21a} to prove
  the claim.  So we may assume that there exists a vertex $v$ in $F$ with $\opp(F,v) \in V(F,-1)$.
  By \autoref{prop:classificationOf0and-1} we may assume $v=e_1$ and (using $\phi:=\phi^F$)
  \begin{align*}
    \begin{aligned}
      V(F,0) \ &= \ \{\phi(e_2)-e_2,\, \phi(e_3)-e_3,\, \ldots, \phi(e_d)-e_d \}\\
      V(F,-1)\ &\subseteq\ \{-e_1, \ldots, -e_d\}\;\cup\; \SetOf{-2e_1-e_r +e_s+e_t}{r,s,t \in [d]
        \text{ p.d.},\, r\ne 1}\,. 
    \end{aligned}
  \end{align*}
  By our assumption the vertex $y:= \opp(F,e_1)$ is at level $-1$ with respect to $F$.  We see that
  it is one of the points
  \begin{align}
    -e_1,  \,,\quad -e_1-e_r+e_s,  \,,\quad \text{or} \quad -2e_1 - e_r + e_s + e_t \,,\label{eq:fourchoices}
  \end{align}
  for pairwise distinct $r,s,t$ and $r\ne 1$. So $u_{F^{(1)}}=\1-2e_1$ for the first two possibilities,
  $u_{F^{(1)}}=\1-e_1$ for the third, and $u_{F^{(j)}}=\1-e_j$ for any $j\ge 2$.

  Let $z$ be the unique vertex in $V(F,-2)$.  If $z_i=\langle u_{F,e_i}, z \rangle$ is positive for
  some $i$, then $z\in V(\neigh(F,e_i),-3)$. But $\neigh(F,e_i)$ is special, so this contradicts
  \autoref{prop:novertexat-3}.  We know that $z\ne -2e_i$ for all $i$ as the vertices must be
  primitive lattice vectors.  Hence, $z=-e_\alpha -e_\beta$ for distinct $\alpha,\beta \in [d]$.  We
  consider all possible forms of $y=\opp(F,e_1)$ according to \eqref{eq:fourchoices}
  separately.\smallskip

  \framebox{Let $y=-e_1$.} Then we have $u_{F^{(1)}}=\1-2e_1$, and this implies that no other vertex
  in $V(F,-1)$ can have a negative first coefficient.  So, up to relabeling,
  $V(F,-1)=\{-e_1,-e_2,\dots, -e_{d-2}\}$.  Choose $k \in \{ 1,d-1,d \}\setminus \{ \alpha,\beta
  \}$.  The fact that $v_P=\0$ implies that $e_k\not\in\image\phi$.  Hence, $P$ is a (possibly skew)
  bipyramid with apices $e_k$ and $\opp(F,e_k)$.\smallskip

  \framebox{Let $y=-e_1-e_r+e_s$.} We may assume that $r=2$ and $s=3$. By \autoref{lem:dd-1d-21b}
  the vertices $e_1$ and $y=\opp(F,e_1)$ are distant, and thus \autoref{lemma:nill_sumDistant} gives
  us that $e_1+y=-e_2+e_3$ is a vertex of $P$.  This means that $\phi(e_2) =
  e_3$. $u_{F^{(1)}}=\1-2e_1$ implies that no other vertex in $V(F,-1)$ can have a negative first
  coefficient. Hence, there are distinct $i,j \in [d]\setminus\{ 1 \}$ such that
  \[
  V(F,-1)\ =\ \SetOf{ -e_k }{ k \in [d] \text{ with } k\ne 1,i,j } \cup \{ y \}\,.
  \]
  If $z_1=0$ then $v_P=\0$ implies $e_1\not\in\image\phi$. In this case $P$ is a skew bipyramid over
  the $(d{-}1)$-polytope $Q=\smallSetOf{x\in P}{x_1=0}$ with apices $e_1$ and $-e_1-e_2+e_3$.

  It remains to consider that case that $\alpha=1$ and $e_1 \in \image\phi$.  This yields
  $\phi(e_\beta)\ne e_1$ as otherwise $-e_\beta =\frac12 ( z + (\phi(e_\beta)-e_\beta) )$ is a non-zero
  lattice point in $P$ which is not a vertex.  This way we obtain
  \begin{equation}\label{eq:sum0:A}
    \begin{split}
      \0\ &=\ v_P \\
      &= \ \sum_{k=1}^d e_k \, + \, \sum_{k=2}^d(\phi(e_k)-e_k) \, - \, \sum_{k\ne 1,i,j} e_k \, + \, y \, + \, z\\
      &= \ \sum_{k=2}^d(\phi(e_k)-e_k) \, + \, e_i + e_j - e_\beta - e_1 - e_2 + e_3 \\
      &= \ \sum_{k=4}^d(\phi(e_k)-e_k) \, + \, \phi(e_3) + e_i + e_j - e_\beta - e_1 - 2e_2 + e_3
      \, ,
    \end{split}
  \end{equation}
  where the  last step uses $\phi(e_2)=e_3$.  Now the vanishing of the third
  coordinate on the right hand side of \eqref{eq:sum0:A} forces $\beta=3$.  This equation then
  simplifies to
  \begin{equation}\label{eq:sum0:B}
    \0 \ = \ \sum_{k=3}^d \phi(e_k) \, + \, e_i + e_j \,-\, \sum_{k\ne 3}e_k  \,-\, e_2 \, .
  \end{equation}
  Note that it is not possible for $i$ or $j$ to be equal to $3$, otherwise the third coordinate of
  the right hand side is strictly positive.  Vanishing of the $i$-th coordinate on the right hand
  side of \eqref{eq:sum0:B} implies the following: If $i \ne 2$ then $i\not\in\image\phi$.  In this
  case $P$ is a skew bipyramid with apices $e_i$ and $\phi(e_i)-e_i$.  The same reasoning holds for
  $j$.  Since $i$ and $j$ are distinct at least one of them must differ from $2$.  So this case
  leads to a skew bipyramid.  \smallskip

  \begin{figure}[tb]
		\begin{tikzpicture}[scale=1.5]
		  \tikzstyle{edge} = [draw,thick,-,black]
		
		  \coordinate (o) at (0,0);
		  \coordinate[label={right:$\scriptstyle y$}] (y) at (1,0);
		  \coordinate[label={below right:$\scriptstyle e_c$}] (ec) at (-60:1);
		  \coordinate[label={below left:$\scriptstyle e_1$}] (e1) at (-120:1);
		  \coordinate[label={left:$\scriptstyle \opp(F,e_c)$}] (oppFec) at (-1,0);
		  \coordinate[label={above right:$\scriptstyle \opp(F^{(1)},e_c)$}] (oppF1ec) at (60:1);
		  \coordinate[label={above left:$\scriptstyle \opp(F^{(c)},e_1)$}] (oppF2e1) at (120:1);
		
		  \draw[edge] (o) -- (oppF2e1) -- (oppFec) -- (e1) -- (ec) -- (y) -- (o) -- (oppFec);
		  \draw[edge] (e1) -- (o) -- (ec);
		  \draw[edge] (o) -- (oppF1ec) -- (y);

		  \foreach \point in {o,v,ec,e1,oppFec,oppF1ec,oppF2e1}
		    \fill[black] (\point) circle (1.3pt);
		
		  \draw (-90:1.732/3) node {$\scriptstyle F$};
		  \draw (-90+60:1.732/3) node {$\scriptstyle F^{(1)}$};
		  \draw (-90-60:1.732/3) node {$\scriptstyle F^{(c)}$};
		\end{tikzpicture} 
    \caption{Facet $F$ and neighboring facets for $y=-2e_1-e_r+e_s+e_t$}
    \label{fig:several_facets}
  \end{figure}
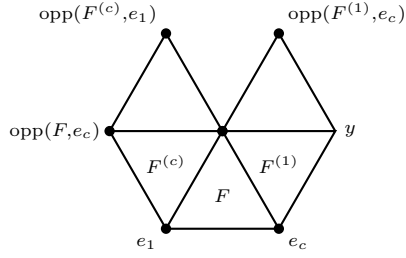

  \framebox{Let $y=-2e_1-e_r+e_s+e_t$.}  We will show that this case cannot occur.  The situation is
  sketched in \autoref{fig:several_facets}. The facet $F^{(1)}:=\neigh(F,e_1)$ has standard facet
  normal $u_{F^{(1)}} = \1 -e_1$.

  Assume first that $\phi(e_r)\ne e_1$.  Then $\opp(F,e_r)=e_i-e_r$ for some $i\ne1,r$, and
  \begin{equation}\label{eq:opp_F_er}
    \langle u_{F^{(1)}}, \opp(F,e_r) \rangle \ = \ \langle \1 - e_1, e_i-e_r \rangle \ = \ 1-1 \ = \ 0 \, .
  \end{equation}
  Furthermore, $u_{F^{(1)},e_r} = e_r - \frac{1}{2} e_1$, and so
  \[
  \langle u_{F^{(1)},e_r}, \opp(F,e_r) \rangle \ = \ \langle e_r-\frac{1}{2}e_1,e_i-e_r\rangle \ =
  \ -1 \, .
  \]
  Since $\langle u_{F^{(1)},e_r}, \opp(F,e_r) \rangle$ is negative,
  \autoref{prop:nill:5.5}.(\ref{prop:nill:5.5:2}) tells us that $\opp(F,e_r)$ is the opposite vertex
  of $e_r$ with respect to $F^{(1)}$, so $\opp(F^{(1)},e_r) = \opp(F,e_r)$.  Using this we conclude
  that $\langle 2\cdot\1 - e_1-2e_r, x\rangle\le 2$ defines the facet of $\neigh(F^{(1)},e_r)$. This
  contradicts the assumption that $P$ is reflexive.  Hence, in the following we can assume that
  $\phi(e_r)=e_1$.

  The facet $F^{(r)}:=\neigh(F,e_r)$ has vertices $\smallSetOf{e_i}{i\ne r}$ and $e_1-e_r$,
  and facet normal $u_{F^{(r)}} = \1 - e_r$.  We want to show that for all vertices $x\ne e_1$ on
  $F^{(r)}$ the opposite vertex $\opp(F^{(r)},x)$ does not coincide with $y$.  Trivially we have
  $\opp(F^{(r)},\opp(F,e_r)) = e_r\ne y$.  It remains to check the vertices $e_i$ for $i\ne r$.  The
  dual basis is
  \begin{align*}
      u_{F^{(r)},\opp(F,e_r)}\ &=\ -e_r\,,&   u_{F^{(r)},e_1}\ &=\ e_1+e_r\,,&   u_{F^{(r)},e_i}\ &=\ e_i\quad\text{ for } i\ne 1,r\,.
  \end{align*}
  so $\langle u_{F^{(r)},e_i}, -2e_1-e_r+e_s+e_t \rangle$ is negative if and only if $i=1$.  Observe
  that $y$ is on level $0$ with respect to $F^{(r)}$.  In view of
  \autoref{prop:nill:5.5}.(\ref{prop:nill:5.5:2}) we obtain that $\opp(F^{(r)},e_i)\ne y$ for $i\ne
  1,r$.

  Now we can apply \autoref{prop:nill:5.5}.(\ref{prop:nill:5.5:3}) which tells us that the vertices
  $e_1$ and $\opp(F,e_1)$ are distant, and according to \autoref{lemma:nill_sumDistant} the point
  $x:=e_1+\opp(F,e_1)=-e_1 -e_r + e_s + e_t$ is a vertex of $P$.  Clearly, $x$ lies on level $0$
  with respect to $F$ which forces $1\in\{s,t\}$ or $r\in\{s,t\}$.  However, both cases are excluded
  due to our initial assumption on $y$ in this case. Hence, this last case does not occur.
\end{proof}

\begin{example}
  As discussed in \autoref{example:B-again} the facet $H$ of the polytope $B$ from
  \autoref{example:eta} has $\eta$-vector $(4,3,2,1)$, and the opposite vertex of $e_3$ is on level
  $-2$ with respect to the facet $H$. Hence, $B$ is a bipyramid.
\end{example}

The following result is dealing with $\eta^{F}=(d,d,d-4,2)$ and reduces it to one case we already
dealt with so far.
\begin{lemma}\label{lem:ddd-42}
  Let $P$ be a $d$-dimensional \str polytope with exactly $3d-2$ vertices and
  some special facet $F$ with $\eta^F = (d,d, d-4, 2)$.  Then $P$ is lattice equivalent to a skew
  bipyramid over a $(d{-}1)$-dimensional smooth Fano polytope with $3d-4$ vertices, or there exists
  another facet $G$ with $\eta^G = (d,d-1,d-2, 1)$ and one vertex $w\in G$ with $\opp(G,w) \not\in
  V(G,0)$.
\end{lemma}

The two cases in the conclusion are not necessarily mutually exclusive.

\begin{proof} Let $\phi:=\phi^F$.  Since $\eta^F_0 = d$ we can assume $F = \conv\{ e_1,e_2,\ldots,
  e_d \}$ by \autoref{prop:classificationOf0and-1}.  Moreover,
  $V(F,0)=\smallSetOf{\phi(e_i)-e_i}{i\in[d]}$ and up to relabeling we can assume that
  $V(F,-1)=\{-e_5,-e_6, \dots, -e_d\}$. Further, if $-e_i$ is a vertex then $\phi(\phi(e_i))=e_i$, which
  implies that $\image(\phi)$ contains the vertices $\{e_5,e_6,\dots,e_d\}$.

  Let $x$ and $y$ be the two vertices in $V(F,-2)$. Then
  \[
  0\ =\ v_P\ =\ x+y \,+\, \sum_{i=1}^d \phi(e_i) \,-\, \sum_{i=5}^d e_i\,,
  \]
  and hence $x+y=-e_h-e_i-e_j-e_k$ for some not necessarily distinct $h,i,j,k\in [d]$.  The equation
  above also shows that $\image \phi = \{ e_h, e_i,e_j,e_k \} \cup \{ e_5, e_6,\ldots, e_d \}$.

  We want to show that all coordinates of $x$ and $y$ are non-positive.  Suppose that $x_a>0$ for
  some $a\in[d]$.  Consider the neighboring facet $F':=\neigh(F,e_a)$.  Then $x$ lies on level $-3$
  or below with respect to $F'$.  However, as $v_P=\0$, the facet $F'$ is special, too, and the case
  $\eta^{F'}_\ell>0$ for $\ell\le-3$ is excluded by \autoref{prop:novertexat-3}.  The same argument
  works for the vertex $y$.  Once more we distinguish two cases.\smallskip

  \noindent\framebox{Let $\{ h,i,j,k \}\ne\{1,2,3,4\}$.}  Without loss of generality we may assume that
  $h,i,j,k$ are all distinct from $1$.  We see that $e_1$ is not contained in $\image(\phi)$.  This
  means that $P$ is a skew bipyramid with apices $e_1$ and $\phi(e_1)-e_1$.
  \smallskip

  \noindent\framebox{Let $\{ h,i,j,k \}=\{1,2,3,4\}$.}  Without loss of generality, $x=-e_1-e_2$ and
  $y=-e_3-e_4$.  In this case the map $\phi$ is bijective since $\image \phi = \{ e_1, e_2, \ldots,
  e_d \}$.  
  We will reduce this to another case that we will have to look into anyway.  This way we save a
  further distinction of cases.  We look at the facet $G:=\neigh(F,e_1)$ with facet normal vector
  $u_{G} = \1-e_1$. Let $e_r$ be the inverse image of $e_1$ with respect to the bijection $\phi$.
  We obtain:
  \[
  \begin{aligned}
    V(G, 1) \ &= \ \SetOf{e_i}{2\le i \le d} \cup \{ \phi(e_1)-e_1 \}\\
    V(G, 0) \ &= \ \SetOf{\phi(e_i)-e_i}{i\in[d]\setminus\{ 1,r \}} \cup \{ e_1 \}\\
    V(G, -1) \ &= \ \SetOf{-e_i}{5\le i \le d} \cup \{ \phi(e_r)-e_r, -e_1-e_2 \}\\
    V(G, -2) \ &= \ \{ -e_3-e_4 \}
  \end{aligned}
  \]
  Therefore $\eta^{G} = (d,d-1,d-2,1)$.  For the vertex $z:=\opp(G,e_r)$ we have the inequality
  $\langle u_{G,e_r},z\rangle<0$ by \autoref{prop:nill:5.5}.(\ref{prop:nill:5.5:2}).  We will show
  that $z$ is not contained in $V(G,0)$.  To this end we compute
  \[
  u_{G,e_r} \ = \
  \begin{cases}
    e_1 + e_r & \text{if $\phi(e_1)=e_r$} \,,\\
    e_r & \text{otherwise}\, .
  \end{cases}
  \]
  Either way, from looking at the above list of vertices of $P$ we see that every vertex $v\in
  V(G,0)$ satisfies $\langle u_{G,e_r}, v\rangle \ge 0$.  We conclude that $z$ is not on level zero
  with respect to $G$, and we may take $w:=e_r$ to prove our claim.
\end{proof}

\begin{lemma}\label{lem:dd-1d-21-oppin0}
  Let $P$ be a \str $d$-polytope with exactly $3d-2$ vertices such that $v_P=\0$. Suppose that $P$
  has a (special) facet $F$ with $\eta^F=(d,d-1,d-2,1)$ and for every vertex $v$ of $F$ we have
  $\opp(F,v) \in V(F,0)$.  Then $P$ is a (possibly skew) bipyramid over some \str polytope of
  dimension $d-1$ with $3d-4$ vertices.
\end{lemma}

\begin{proof}
  By \autoref{prop:classificationOf0and-1}\eqref{prop:classificationOf0and-1:2} we may assume that
  $F=\conv\{e_1,e_2, \ldots, e_d\}$ and
  \begin{align}\label{eq:VF-1pre}
    \begin{split}
      V(F,0)\ &=\ \{-e_1-e_2+e_a+e_b,\, \phi(e_3)-e_3,\, \ldots,\, \phi(e_d)-e_d\}\\
      V(F,-1)\ &\subseteq\ \{-e_1,\, -e_2,\, \ldots,\, -e_d\} \;\cup\; \SetOf{-e_1-e_2 + e_k}{k\in
        [d]}
    \end{split}
  \end{align}
  for $a, b \in [d] \setminus\{ 1,2 \}$ not necessarily distinct.

  Let $z=\sum_{i\ge 1}\mu_ie_i$ be the unique vertex at level $-2$ with respect to $F$.  If
  $\mu_i=\langle u_{F,e_i}, z \rangle$ is positive for some $i$, then $z\in
  V(\neigh(F,e_i),-3)$. But $\neigh(F,e_i)$ is special, so this contradicts
  \autoref{prop:novertexat-3}. Hence, $\mu_i\le 0$ for all $i$ and $z=-e_r-e_s$ for some indices
  $r\ne s$. Furthermore, if $-e_1-e_2+e_k\in P$ for some $k\ne r,s$ we get that this vertex as well
  as $z$ would lie on level $-2$ with respect to the facet $\neigh(F,e_k)$. So this adjacent facet
  must have $\eta$-vector $(d,d,d-4,2)$ which is already dealt with \autoref{lem:ddd-42}. This implies
  \begin{equation}\label{eq:VF-1}
    V(F,-1) \ \subset \ \{-e_1,\, -e_2,\, \ldots,\, -e_d\}\;\cup\;\{-e_1-e_2+e_r\,,\ -e_1-e_2+e_s\}\,.
  \end{equation}

  Let $H:=\neigh(F,e_a)$.  The vertex $e_a$ is good due to our assumption every vertex of $F$ is good, and so
  $u_{H}=\1-e_a$.  We distinguish between $a=b$ and $a\ne b$.\smallskip

  \noindent\framebox{Let $a=b$.} Then $x=\opp(F,e_1)=\opp(F,e_2)=-e_1-e_2+2e_a$ and we compute $\langle
  u_H,x\rangle=\langle\1-e_a,x\rangle=-2$.  Either $x$ is the only vertex in $V(H,-2)$ and hence $\eta^H =
  (d,d-1,d-2,1)$ or there is another vertex in $V(H,-2)$ besides $x$ which would result in $\eta^H = (d,d,d-4,2)$,
  but this case is already captured in \autoref{lem:ddd-42}. So we can assume that $x$ is the only vertex in
  $V(H,-2)$.

  Up to exchanging $e_1$ and $e_2$, we can assume that $\phi(e_a)\ne e_1$.  We consider the vertex
  $y:=\opp(H,e_1)$. The case $y \not\in V(H,0)$ is already covered in \autoref{prop:dd-1d-21}. So we can assume
  $y$ is in $V(H,0)$. As $u_{H,e_1}=e_1$, this implies that $\langle u_{H,e_1}, y\rangle < 0$, that is, in the
  basis defined by $H$, the $e_1$-coordinate of $y$ must be negative. As $\phi(e_a)\ne e_1$ the $e_1$-coordinate
  of $y$ in the basis of $H$ coincides with $y_1$, which is the $e_1$-coordinate in the basis of $F$.  We can
  check which vertices or vertex candidates have a negative first coordinate.  These turn out to be $x$, and
  $-e_1$, $-e_1-e_2+e_r$, and $-e_1-e_2+e_s$, if they exist.  However, none of these lies on level zero.  Hence, a
  vertex $y$ with these properties does not exist in $V(F,0)$.  This finishes the case that $a=b$.\smallskip

  \noindent\framebox{Let $a\ne b$.} Then our case distinction has a further ramification.
  \begin{enumerate}
    \item Let $\{ a,b \} \neq \{ r,s \}$.  Without loss of generality we can assume that $a \not\in\{ r,s \}$. We
      define $F^{(a)} := \neigh(F,e_a)$. Since $\langle u_{F,e_a}, -e_r-e_s \rangle = 0$ we know that $-e_r-e_s$
      is contained in $V(F^{(a)}, -2)$.  In particular, $\eta^{F^{(a)}}_{-2}$ does not vanish and hence the
      $\eta$-vector of $F^{(a)}$ reads $(d,d-1,d-2,1)$ or $(d,d,d-4,2)$. The latter being already discussed in
      \autoref{lem:ddd-42} so we are left with the first $\eta$-vector. This further implies that $\langle
      u_{F,e_a}, x \rangle \le 0$ for all $x\in V(F,-1)$. All vertices in $F$ are good, so
      \autoref{lem:phi:notwonegative} tells us that there can only be one vertex $x\in V(F,-1)$ with $\langle
      u_{F,e_a}, x \rangle = -1$. Due to the bounds obtained from \autoref{lemma:2} the only possible vertex in
      $V(F,-1)$ which has a non-positive scalar product with $u_{F,e_a}$ is $-e_a$. Summing up all vertices $y$ of
      $P$ with $\langle u_{F,e_a}, y \rangle \ne 0$ implies that $e_a$ cannot be contained in the image of $\phi$.
      Hence,
      \[
      V(F^{(a)},0) \ = \ \SetOf{\phi(e_i)-e_i}{i \in [d] \setminus \{ 1,2,a \}} \cup \{ \pm e_a \} \, .
      \]
      Up to exchanging $e_1$ and $e_2$ we may assume that $\phi(e_a)\ne e_1$.  Thus $u_{F^{(a)}, e_1} = e_1$,
      which implies $\langle u_{F^{(a)}, e_1}, x \rangle \ge 0$ for all $x\in V(F^{(a)},0)$. However, this
      excludes that $\opp(F^{(a)},e_1)$ is contained in $V(F^{(a)},0)$.  Which is the case already considered in
      \autoref{prop:dd-1d-21}.

    \item Let $\{ a,b \} = \{r,s\}$.  Then $x=-e_1-e_2 +e_r+e_s$ is the vertex which is opposite both
      to $e_1$ and $e_2$ with respect to $F$.  From the argument in the case $a=b$ we learn that if $-e_1-e_2
      +e_r$ or $-e_1-e_2 +e_s$ were vertices of $P$ then the facet $F^{(r)}:=\neigh(F,e_r)$ or the facet
      $F^{(s)}:=\neigh(F,e_s)$ would contain a vertex such that its opposite does not lie on level zero with
      respect to $F^{(r)}$ or $F^{(s)}$. This vertex would either be $e_1$ or $e_2$.  Again this case is already
      captured in \autoref{prop:dd-1d-21}.  Therefore, neither $-e_1-e_2 +e_r$ nor $-e_1-e_2 +e_s$ is a vertex of
      $P$.  By \eqref{eq:VF-1} we have that $V(F,-1)\subset\SetOf{-e_i}{i\in [d]}$.

      Using the generic forms of all vertices given in \eqref{eq:VF-1pre} and
      \eqref{eq:VF-1} we conclude that there is no $v\in \Vert P$ such that $x+v$ is again a
      vertex. Hence, by \autoref{lemma:nill_sumDistant}, no vertex of $P$ is distant to $x$.  In
      particular, there is a facet $G$ which contains both $x=-e_1-e_2+e_r+e_s$ and $z=-e_r
      -e_s$. This implies
      \[
      2 \ = \ \langle u_G, -e_1-e_2+e_r+e_s \rangle + \langle u_G, -e_r-e_s \rangle \ = \ \langle  u_G, -e_1 -e_2 \rangle\, .
      \]
      However, this inequality says that if $-e_1$ and $-e_2$ are vertices of $P$ then they, too, are
      contained in $G$.  The four points $-e_1,-e_2,x,z$ are linearly dependent, hence the facet $G$
      cannot be a simplex.  This is impossible.

      We may thus assume that $-e_1$ is not a vertex of $P$.  The vertices of $P$ with a non-vanishing
      first coordinate comprise $e_1$, $-e_1-e_2+e_r+e_s$ and possibly $e_1-e_i$ for some $i\ge 3$.
      Since, however, the vertex sum $v_P$ vanishes the sum of the first coordinates of all vertices
      vanishes.  This excludes that $e_1-e_i$ is a vertex of $P$ for any $i\ge 3$; this is to say that
      $e_1$ is not contained in the image of $\phi$.

      Suppose that  $\phi(e_r)=e_s$.  Then we have
      \[
      \frac{1}{2}(\phi(e_r)-e_r+z) \ = \ \frac{1}{2}(e_s - e_r - e_r - e_s) \ = \ -e_r \, .
      \]
      In view of the fact, however, that both $\phi(e_r)-e_r$ and $z$ are vertices of $P$ this is a
      contradiction to the terminality of $P$.  So $\phi(e_r)\ne e_s$ and, by symmetry, also
      $\phi(e_s)\ne e_r$.

      The standard normal vector of the facet $F^{(r)}$ reads $\1-e_r$, and thus the vertex
      $\phi(e_s)-e_s$ lies on level zero with respect to $F^{(r)}$.  Each such vertex is opposite to
      some vertex in $F^{(r)}$, and we infer that $\opp(F^{(r)},e_s)=\opp(F,e_s)=\phi(e_s)-e_s$.  This
      is to say that the vertex $e_s$ is good with respect to $F^{(r)}$.

      Up to exchanging $e_r$ and $e_s$ we may assume that $e_2 \ne \phi(e_r)$.  We claim that
      $\phi(e_r)$ is also good with respect to $F^{(r)}$.  To see this we consider two cases. Either
      $\phi(\phi(e_r)) = e_r$ or $\phi(\phi(e_r)) \ne e_r$.  In the first case, $-e_r$ must be a
      vertex of $P$, and $-e_r = (\phi(e_r) - e_r) - \phi(e_r)$ lies on level zero with respect to
      $F^{(r)}$.  This yields $\opp(F^{(r)},\phi(e_r))=-e_r$, and $\phi(e_r)$ is good with respect to $F^{(r)}$.

      It remains to consider the case where $\phi(\phi(e_r))\ne e_r$.  Then $\opp(F,\phi(e_r)) =
      \opp(F^{(r)},\phi(e_r))\in V(F,0)$ which makes $\phi(e_r)$ good with respect to $F^{(r)}$.

      The vertex $-e_r-e_s = (\phi(e_r)-e_r) - \phi(e_r) -e_s$ now satisfies
      \[
      \langle  u_{F^{(r)}, \phi(e_r)}, -e_r-e_s \rangle \ = \ -1 \ = \ \langle  u_{F^{(r)}, e_s}, -e_r-e_s \rangle\, .
      \]
      The final contradiction now comes from \autoref{lem:phi:notwonegative} since we showed that both $e_s$ and      $\phi(e_r)$ are good with respect to $F^{(r)}$ and $\opp(F^{(r)},e_s) = \phi(e_s)-e_s \ne \opp(F^{(r)},\phi(e_r))$.\qedhere
  \end{enumerate}
\end{proof}

The following compiles the results of this section into one concise statement.

\begin{theorem}\label{thm:finalsummary}
  Let $P$ be a $d$-dimensional \str polytope with exactly $3d-2$ vertices and $v_P=\0$. Then $P$ is
  lattice equivalent to one of the following:
  \begin{enumerate}
  \item a (possibly skew) bipyramid over
    \begin{itemize}
    \item either $P_5 \oplus P_6^{\oplus \frac{d-3}{2}}$
    \item or a (possibly skew) bipyramid over $P_6^{\oplus \frac{d-2}{2}}$,
    \end{itemize}
  \item $\HSBC{4} \oplus P_6^{\oplus \frac{d}{2} - 2}$.
  \end{enumerate}
\end{theorem}
\begin{proof}
  According to \autoref{tab:eta} the a priori feasible $\eta$-vectors are $(d,d,d-3,0,1)$,
  $(d,d-2,d)$, $(d,d,d-4,2)$, or $(d,d-1,d-2,1)$.

  The first case is actually excluded by \autoref{prop:novertexat-3}.  If the $\eta$-vector of each
  facet of $P$ equals $(d,d-2,d)$, then $P$ is centrally symmetric by
  \autoref{prop:centrallySymmetric} and thus classified by \autoref{cor:centrally_symmetric}.  This,
  in particular, comprises the polytope $\HSBC{4} \oplus P_6^{\oplus \frac{d}{2} - 2}$.  Therefore,
  our goal is to show that in all remaining cases $P$ is a possibly skew bipyramid over a suitable
  base polytope.

  We may assume that $P$ has at least one facet, $F$, with $\eta^F=(d,d,d-4,2)$ or
  $\eta^F=(d,d-1,d-2,1)$.  Recall that all facets are special as $v_P=\0$.

  Let us consider the case where $\eta^F=(d,d-1,d-2,1)$.  Then our case distinction further ramifies
  depending on where the vertices opposite to the vertices of $F$ are located.  Let
  \[
  S \ = \ \SetOf{\opp(F,v)}{v\in F} \,.
  \]
  Either there is a vertex in $S$ which lies at level $-1$ or $-2$ with respect to $F$. Then we call
  $F$ of \emph{type A}, and Proposition~\ref{prop:dd-1d-21} tells us that $P$ is a proper or skew
  bipyramid.  Or all vertices in $S$ are at level zero.  We will postpone the latter case, which we
  call \emph{type B}.

  Suppose now that $\eta^F=(d,d,d-4,2)$.  Then \autoref{lem:ddd-42} shows that either $P$ is
  possibly skew bipyramid, or there is another facet whose $\eta$-vector reads $(d,d-1,d-2,1)$, and
  which is of type A.  So this is already resolved.

  It remains to consider $\eta^F=(d,d-1,d-2,1)$ of type B.  Then \autoref{lem:dd-1d-21-oppin0}
  finally shows that $P$ again must be a proper or skew bipyramid. Our argument is based on the
  previous cases.
\end{proof}

\providecommand{\bysame}{\leavevmode\hbox to3em{\hrulefill}\thinspace}
\providecommand{\MR}{\relax\ifhmode\unskip\space\fi MR }
\providecommand{\MRhref}[2]{%
  \href{http://www.ams.org/mathscinet-getitem?mr=#1}{#2}
}
\providecommand{\href}[2]{#2}


\begin{thebibliography}{10}

\bibitem{Batyrev91}
Victor~V. Batyrev, \emph{On the classification of smooth projective toric
  varieties}, Tohoku Math. J. (2) \textbf{43} (1991), no.~4, 569--585, doi:
  \href{http://dx.doi.org/10.2748/tmj/1178227429}{10.2748/tmj/1178227429}.
  \MR{1133869 (92j:14065)}

\bibitem{Batyrev94calabi}
Victor~V. Batyrev, \emph{Dual polyhedra and mirror symmetry for {C}alabi--{Y}au
  hypersurfaces in toric varieties}, J. Alg. Geom (1994), 493--535.

\bibitem{Batyrev2007}
Victor~V. Batyrev, \emph{On the classification of toric {F}ano {$4$}-folds}, J.
  Math. Sci. (New York) \textbf{94} (1999), no.~1, 1021--1050, doi:
  \href{http://dx.doi.org/10.1007/BF02367245}{10.1007/BF02367245}.
  \MR{MR1703904 (2000e:14088)}

\bibitem{BB}
Victor~V. Batyrev and Lev~A. Borisov, \emph{Mirror duality and
  string--theoretic {Hodge} numbers}, Inventiones Math. \textbf{126} (1996),
  183--203,
  \href{http://arxiv.org/abs/alg-geom/9509009}{arxiv:alg-geom/9509009}.

\bibitem{GRDB}
Gavin Brown and Alexander Kasprzyk, \emph{Graded ring data base}, 2009--2012.

\bibitem{Casagrande03}
Cinzia Casagrande, \emph{Centrally symmetric generators in toric {F}ano
  varieties}, Manuscripta Math. \textbf{111} (2003), no.~4, 471--485.
  \MR{2002822 (2004k:14087)}

\bibitem{Casagrande06}
\bysame, \emph{The number of vertices of a {F}ano polytope}, Ann. Inst. Fourier
  (Grenoble) \textbf{56} (2006), no.~1, 121--130, doi:
  \href{http://dx.doi.org/10.5802/aif.2175}{10.5802/aif.2175}. \MR{2228683
  (2007c:14053)}

\bibitem{Book_CoxLittleSchenck}
David~A. Cox, John~B. Little, and Henry~K. Schenck, \emph{Toric varieties},
  Graduate Studies in Mathematics, vol. 124, American Mathematical Society,
  Providence, RI, 2011. \MR{2810322 (2012g:14094)}

\bibitem{Ewald88}
G{\"u}nter Ewald, \emph{On the classification of toric {F}ano varieties},
  Discrete Comput. Geom. \textbf{3} (1988), no.~1, 49--54, doi:
  \href{http://dx.doi.org/10.1007/BF02187895}{10.1007/BF02187895}. \MR{918178
  (88m:14037)}

\bibitem{Ewald96}
\bysame, \emph{Combinatorial convexity and algebraic geometry}, Graduate Texts
  in Mathematics, vol. 168, Springer-Verlag, New York, 1996. \MR{1418400
  (97i:52012)}

\bibitem{DMV:polymake}
Ewgenij Gawrilow and Michael Joswig, \emph{polymake: a framework for analyzing
  convex polytopes}, Polytopes---combinatorics and computation (Oberwolfach,
  1997), DMV Sem., vol.~29, Birkh\"auser, Basel, 2000, pp.~43--73. \MR{1785292
  (2001f:52033)}

\bibitem{KN5}
Maximilian Kreuzer and Benjamin Nill, \emph{Classification of toric {F}ano
  5-folds}, Adv. Geom. \textbf{9} (2009), no.~1, 85--97, doi:
  \href{http://dx.doi.org/10.1515/ADVGEOM.2009.005}{10.1515/ADVGEOM.2009.005}.
  \MR{2493263 (2010a:14077)}

\bibitem{KreuzerSkarke3d}
Maximilian Kreuzer and Harald Skarke, \emph{On the classification of reflexive
  polyhedra}, Comm. Math. Phys. \textbf{185} (1997), no.~2, 495--508, doi:
  \href{http://dx.doi.org/10.1007/s002200050100}{10.1007/s002200050100}.
  \MR{1463052 (98f:32029)}

\bibitem{KreuzerSkarke4d}
\bysame, \emph{On the classification of reflexive polyhedra in four
  dimensions}, Advances Theor. Math. Phys. \textbf{4} (2002), 1209--1230,
  \href{http://arxiv.org/abs/hep-th/0002240}{arXiv:hep-th/0002240}.

\bibitem{smoothreflexive}
Benjamin Lorenz and Andreas Paffenholz, \emph{Smooth reflexive polytopes up to
  dimension $9$}, 2008.

\bibitem{1067.14052}
Benjamin Nill, \emph{Gorenstein toric {Fano} varieties}, Manuscr. Math.
  \textbf{116} (2005), no.~2, 183--210, doi:
  \href{http://dx.doi.org/10.1007/s00229-004-0532-3}{10.1007/s00229-004-0532-3}.

\bibitem{0511294}
\bysame, \emph{Classification of pseudo-symmetric simplicial reflexive
  polytopes}, Algebraic and geometric combinatorics, Contemp. Math., vol. 423,
  Amer. Math. Soc., Providence, RI, 2006, doi:
  \href{http://dx.doi.org/10.1090/conm/423/08082}{10.1090/conm/423/08082},
  pp.~269--282. \MR{2298762 (2008b:52017)}

\bibitem{NillObro10}
Benjamin Nill and Mikkel {\O}bro, \emph{{$\mathbb{Q}$}-factorial {G}orenstein toric
  {F}ano varieties with large {P}icard number}, Tohoku Math. J. (2) \textbf{62}
  (2010), no.~1, 1--15, doi:
  \href{http://dx.doi.org/10.2748/tmj/1270041023}{10.2748/tmj/1270041023}.
  \MR{2654299 (2011f:14085)}

\bibitem{OebroPhD}
Mikkel {\O}bro, \emph{Classification of smooth {Fano} polytopes}, Ph.D. thesis,
  University of Aarhus, 2007, available at
  \url{https://pure.au.dk/portal/files/41742384/imf_phd_2008_moe.pdf}.

\bibitem{Obro08}
Mikkel {\O}bro, \emph{Classification of terminal simplicial reflexive
  {$d$}-polytopes with {$3d-1$} vertices}, Manuscripta Math. \textbf{125}
  (2008), no.~1, 69--79, doi:
  \href{http://dx.doi.org/10.1007/s00229-007-0133-z}{10.1007/s00229-007-0133-z}.
  \MR{2357749 (2008h:52013)}

\bibitem{Voskresenskij1985}
Valentin~E. Voskresenski{\u\i} and Aleksandr~A. Klyachko, \emph{Toric {F}ano
  varieties and systems of roots}, Izv. Akad. Nauk SSSR Ser. Mat. \textbf{48}
  (1984), no.~2, 237--263, doi:
  \href{http://dx.doi.org/10.1070/IM1985v024n02ABEH001229}{10.1070/IM1985v024n02ABEH001229}.
  \MR{740791 (85k:14024)}

\end{thebibliography}
\end{document}